\documentclass[12pt,reqno]{amsart}

\newcommand\version{November 21, 2011}

\usepackage{amsmath,amsfonts,amsthm,amssymb,amsxtra}

\newtheorem{theorem}{Theorem}[section]
\newtheorem{proposition}[theorem]{Proposition}
\newtheorem{lemma}[theorem]{Lemma}
\newtheorem{corollary}[theorem]{Corollary}

\theoremstyle{definition}

\theoremstyle{remark}

\numberwithin{equation}{section}

\newcommand{\C}{\mathbb{C}}

\newcommand{\const}{\mathrm{const}\ }

\renewcommand{\epsilon}{\varepsilon}

\newcommand{\Hei}{\mathbb{H}}
\newcommand{\loc}{{\rm loc}}

\newcommand{\N}{\mathbb{N}}

\renewcommand{\phi}{\varphi}
\newcommand{\R}{\mathbb{R}}

\newcommand{\Sph}{\mathbb{S}}

\DeclareMathOperator{\im}{Im}

\DeclareMathOperator{\re}{Re}

\begin{document}

\title[Sharp inequalities --- \version]{Sharp constants in several inequalities on the Heisenberg group}

\author{Rupert L. Frank}
\address{Rupert L. Frank, Department of Mathematics, Princeton University, Washington Road, Princeton, NJ 08544, USA}
\email{rlfrank@math.princeton.edu}

\author{Elliott H. Lieb}
\address{Elliott H. Lieb, Departments of Mathematics and Physics, Princeton University, P.~O.~Box 708, Princeton, NJ 08544, USA}
\email{lieb@math.princeton.edu}

\thanks{\copyright\, 2011 by the authors. This paper may be reproduced, in its entirety, for non-commercial purposes.\\
Support by U.S. NSF grant PHY 1068285 (R.L.F.) and PHY 0965859 (E.H.L.) is gratefully acknowledged.}

\begin{abstract}
 We derive the sharp constants for the inequalities on the Heisenberg
group $\Hei^n$ whose analogues on Euclidean space $\R^n$ are the well known
Hardy-Little\-wood-Sobolev inequalities. Only one special case had been
known previously, due to  Jerison-Lee more than twenty years ago. From
these inequalities we obtain the sharp constants for their duals, which
are the Sobolev inequalities for the Laplacian and conformally invariant
fractional Laplacians. By considering
limiting cases of these inequalities sharp constants for
the analogues of the Onofri and log-Sobolev inequalities on $\Hei^n$ are obtained. The
methodology is completely different from that used to obtain the $\R^n$
inequalities and can be (and has been) used to give a new, rearrangement free, proof
of the HLS inequalities.
\end{abstract}

\maketitle

\section{Introduction}

We shall be concerned with sharp constants in some classical integral inequalities on the Heisenberg group. These have analogues on $\R^n$, known as the Hardy-Littlewood-Sobolev inequalities, and their limiting version, the logarithmic Hardy-Littlewood-Sobolev inequality. They appear in many areas of analysis, often in their dual forms as Sobolev inequalities and Onofri inequalities.

The Hardy-Littlewood-Sobolev inequality \cite{HaLi1,HaLi2,So} on $\R^N$ is
\begin{equation} \label{hls}
\left| \iint_{\R^N\times\R^N} \frac{\overline{f(x)}\ g(y)}{|x-y|^\lambda} \,dx\,dy \right|
\leq D_{N,\lambda}^{\text{HLS}} \ \|f\|_p \ \|g\|_p
\end{equation}
for $0< \lambda <N$ and $p= 2N/(2N-\lambda)$. The sharp constant
$D_{N,\lambda}^{\text{HLS}}$ was obtained in \cite{Li} by utilizing the conformal
symmetries of (\ref{hls}) and symmetric decreasing rearrangements.
(See \cite{LiLo} for a discussion of these conformal symmetries and
\cite{CaLo} for a different, ingenious proof, and see \cite{FrLi}
for the connection between \eqref{hls} and reflection-positivity.)

A point $p\in \R^N$ can also be viewed as a member of the translation
group (with $p: \ x\mapsto x+p$), from which point of view $|x-y|$ becomes
$|x^{-1} y|$ and $dx$ is Haar measure on the group.  Then, (\ref{hls})
becomes an inequality for functions on this group.

An inequality similar to \eqref{hls} holds for the Heisenberg group $\Hei^n$ and, to
our knowledge, originates in the work of Folland and Stein \cite{FoSt}. For $0<\lambda<Q=2n+2$ and $p=2Q/(2Q-\lambda)$,
\begin{equation} \label{eq:hlshei}
\left| \iint_{\Hei^n\times\Hei^n} \frac{\overline{f(u)}\ g(v)}{|u^{-1}v|^\lambda} \,du\,dv \right|
\leq D_{n,\lambda} \ \|f\|_p \ \|g\|_p \,.
\end{equation}
Here $u^{-1}v$ is the group product, $|\cdot|$ is the homogeneous norm and $du$ is Haar measure. Details of this notation will be explained in the next section. The kernel in \eqref{eq:hlshei}, like that in \eqref{hls}, is positive definite, so it suffices to verify these inequalities for $f=g$.

The work \cite{FoSt} proves the existence of a finite constant $D_{n,\lambda}$ in \eqref{eq:hlshei}, but leaves open the question of its optimal value. There is a natural guess \cite{BrFoMo} for an optimizing function,
\begin{equation}
 \label{eq:opt}
H(u)= \left((1+|z|^2)^2+t^2\right)^{-(2Q-\lambda)/4}\,, 
\end{equation}
where $u=(z,t)$ in the identification of $\Hei^n$ with $\C^n\times\R$. This is in analogy with the optimizers in the Euclidean inequality \eqref{hls}, but we note the subtlety that the level surfaces of $H$ are neither isoperimetric surfaces \cite{CaDaPaTy} nor level surfaces of the homogenous norm $|\cdot|$ appearing in \eqref{eq:hlshei}. This disparity is connected with the fact that the $2n+1$ real coordinates parametrizing $\Hei^n$ do not all appear to the
same degree, as do the $\R^N$ coordinates in the norm $|x|^2 =\sum x_i^2$ appearing in \eqref{hls}. Consequently, arguments involving symmetric decreasing rearrangements can \emph{not} be expected to work for $\Hei^n$, and thus the sharp constant evaluation in \eqref{eq:hlshei} is considerably more difficult than in \eqref{hls}.

Nevertheless, in a celebrated paper \cite{JeLe}, Jerison and Lee were able to prove that the function $H$ in \eqref{eq:opt} is an optimizer in the special case $\lambda=Q-2$. (Actually, they solved the problem in the dual formulation of a Sobolev inequality involving the sub-Laplacian on $\Hei^n$.) Another reason to believe the correctness of $H$ is that the endpoint case, with $|u^{-1}v|^{-\lambda}$ replaced by $\log|u^{-1}v|$, has recently been settled \cite{BrFoMo} and the function $H$ with $\lambda=0$ turns out to be the optimizer there, too. Some other recent, related works on sharp constants are \cite{CoLu1,CoLu2}.

In this paper we evaluate the sharp constant $D_{n,\lambda}$ in the $\Hei^n$ case for \emph{all} allowed values of $\lambda$ and we show that, as in the case of the HLS inequality, $H$ is the unique optimizer, up to translations and dilations.
The $\lambda =Q-2$ case is special and we prove it separately in Section \ref{sec:jl} in a manner much simpler than either \cite{JeLe} or the general $\lambda$ proof in the rest of our paper.

We must first show that there is an optimizer for the inequality, i.e., that there is a pair $f$ and $g$ that actually gives equality in \eqref{eq:hlshei} with the sharp constant. We first show that the kernel $|u^{-1}v|^{-\lambda}$ is positive definite, which implies that we can restrict our search to $f=g$. The positive definiteness is not as obvious here as it is in the $\R^N$ inequality \eqref{hls}. Indeed, the operator square root of $|u^{-1}v|^{-\lambda}$ is not $|u^{-1}v|^{-(Q+\lambda)/2}$, as one might guess on the basis of the Euclidean case. The two are closely related, however, and, with the aid of a multiplier theorem of \cite{MuRiSt}, we can estimate the `true' square root in terms of the `false' square root.

The existence proof is more involved than the analogous proof for (\ref{hls}), because it is unclear how the left side of \eqref{eq:hlshei} behaves under any kind of rearrangement. We use a relatively recent, sophisticated version of the Sobolev inequality which originates in the work of several authors \cite{GeMeOr,BaGa,BaGeXu}. This inequality was used by G\'erard \cite{Ge} to prove the existence of an optimizer in the $\R^N$ Sobolev inequality, see also \cite{KiVi}. Our existence proof is accomplished with a dual form of the corresponding $\Hei^n$ inequality together with the extended Fatou lemma in \cite{BrLi}, thereby shortening the proof relative to \cite{Ge,KiVi}.

The final, but most complicated task is to evaluate the optimizer. We do this by examining the second variation inequality. Using an idea of Chang and Yang \cite{ChYa}, which expands an argument of Hersch \cite{He}, we show that the purported inequality is, in fact, an inequality in the opposite direction; the only function for which both inequalities are true is the stated function $H$. This step is most conveniently carried out in framework of the complex sphere $\Sph^{2n+1}$ where $H$ becomes the constant function. It is on $\Sph^{2n+1}$ that one easily sees a natural way to break the huge symmetry group of the inequality by requiring that the center of mass of the function be zero. The use of the complex sphere $\Sph^{2n+1}$ is not unlike the use of the real sphere in \cite{Li}.

It is well known that one can achieve new, useful inequalities by differentiating \eqref{hls} at the endpoints $\lambda=0$ and $\lambda=N$. The former case yields `logarithmic HLS inequalities', with sharp constants, going back to \cite{CaLo2,Be2}. The dual of these inequalities is Onofri's inequality and its generalizations, also with sharp constants; see \cite{On,ChYa,OsPhSa} and references therein. Differentiation at $\lambda=N$ yields a sharp logarithmic Sobolev inequality \cite{Be1}. We are able to do the parallel calculations for \eqref{eq:hlshei}. We rederive the result of \cite{BrFoMo} mentioned above by differentiating our sharp bound at $\lambda=0$ and thereby giving another proof of \cite{BrFoMo}. At the other endpoint, $\lambda=Q$, differentiation of \eqref{eq:hlshei} yields what appears to be a new logarithmic Sobolev inequality on $\Hei^n$.

As we said, the Heisenberg group proof is considerably more complicated than the proof of HLS in Euclidean space because it does not use symmetrization. The proof we give here will thus work as well, \emph{mutatis mutandis}, for \eqref{hls} and provides the first symmetrization-free proof of HLS for the entire range of $\lambda$ \cite{FrLi2}. See also \cite{FrLi} for a different symmetrization-free proof for $\lambda\geq N-2$.

Another area to which our methods seem applicable are the groups of Heisenberg type, in which the variable $t$ becomes multi-dimensional; see, e.g., \cite{GaVa} for a sharp inequality for partially symmetric functions related to \cite{JeLe1}.

Finally, we mention that
many computations with Jacobi polynomials are needed; we leave it as an
open problem to find an essential simplification of our computations.


\section{Main result}

The Heisenberg group $\Hei^n$ is $\C^n\times\R$ with elements $u=(z,t)$ and group law
$$
u u' = (z,t) (z',t') = (z+z', t+t'+2\im z\cdot\overline{z'}) \,.
$$
Here we have set $z\cdot\overline{z'} = \sum_{j=1}^n z_j \overline{z_j'}$. Haar measure on $\Hei^n$ is the usual Lebesgue measure $du=dz\,dt$. (To be more precise, $dz=dx\,dy$ if $z=x+iy$ with $x,y\in\R^n$.) We write $\delta u = (\delta z,\delta^2 t)$ for dilations of a point $u=(z,t)$ and denote the homogeneous norm on $\Hei^n$ by
$$
| u | = |(z,t)| = (|z|^4+t^2)^{1/4} \,.
$$
As usual, we denote the homogeneous dimension by $Q:=2n+2$.

We shall prove

\begin{theorem}\label{main}
 Let $0<\lambda<Q=2n+2$ and $p:=2Q/(2Q-\lambda)$. Then for any $f,g\in L^p(\Hei^n)$
\begin{equation}
 \label{eq:main}
\left| \iint_{\Hei^n\times\Hei^n} \frac{\overline{f(u)}\ g(v)}{|u^{-1} v|^\lambda} \,du\,dv \right| 
\leq \left(\frac{\pi^{n+1}}{2^{n-1} n!} \right)^{\lambda/ Q} \frac{n!\, \Gamma((Q-\lambda)/2)}{\Gamma^2((2Q-\lambda)/4)} \ \|f\|_p \|g\|_p
\end{equation}
with equality if and only if
$$
f(u) = c \ H(\delta(a^{-1} u)) \,,
\qquad
g(u) = c' \ H(\delta(a^{-1} u))
$$
for some $c,c'\in\C$, $\delta>0$ and $a\in\Hei^n$ (unless $f\equiv 0$ or $g\equiv 0$). Here $H$ is the function in \eqref{eq:opt}.
\end{theorem}

In other words, we prove that the function $H$ in \eqref{eq:opt} is the unique optimizer in inequality \eqref{eq:hlshei} up to translations, dilations and multiplication by a constant. An equivalent characterization of all optimizers is the form
$$
f(z,t) = \frac{c}{\left|i|z|^2 +t+2iz\cdot\overline w +\mu\right|^{(2Q-\lambda)/2}}
$$
with $c,\lambda\in\C$ and $w\in\C^n$ satisfying $\im \mu>|w|^2$, and $g$ proportional to $f$.

By a duality argument, based on the fact (see \cite{FoSt} and \cite[(XIII.26)]{St})  that the Green's function of the sub-Laplacian $\mathcal L$ in \eqref{eq:sublap} is $2^{n-2} \Gamma^2(n/2)\pi^{-n-1}|u|^{-Q+2}$, we see that the case $\lambda=Q-2$ of Theorem~\ref{main} is equivalent to the sharp Sobolev inequality \cite{JeLe} of Jerison and Lee,
\begin{equation}
 \label{eq:sobjl}
\frac14 \sum_{j=1}^n \int_{\Hei^n} \left( \left| \left( \frac\partial{\partial x_j} +2 y_j \frac\partial{\partial t} \right) u \right|^2 + \left| \left( \frac\partial{\partial y_j} -2 x_j \frac\partial{\partial t} \right) u\right|^2 \right) \,du
\geq \frac{\pi n^2 }{(2^{2n} n!)^{1/(n+1)} } \|u\|_q^2
\end{equation}
with $q=2Q/(Q-2)$. We shall give a short, direct proof of \eqref{eq:sobjl} in Section~\ref{sec:jl} below, that is easier than going the route of \eqref{eq:main}.

The Cayley transform $\mathcal C$, the explicit definition of which will be recalled in Appendix \ref{sec:equiv}, defines a bijection between the Heisenberg group $\Hei^n$ and the punctured sphere $\Sph^{2n+1}\setminus\{(0,\ldots,0,-1)\}$. We consider the sphere $\Sph^{2n+1}$ as a subset of $\C^{n+1}$ with coordinates $(\zeta_1,\ldots,\zeta_{n+1})$ satisfying $\sum_{j=1}^{n+1} |\zeta_j|^2 =1$, and (non-normalized) measure denoted by $d\zeta$. Again we shall use the notation $\zeta\cdot\overline\eta = \sum_{j=1}^{n+1} \zeta_j\overline{\eta_j}$ for the scalar product induced by $\C^{n+1}$. Via this transform Theorem \ref{main} is equivalent to

\begin{theorem}\label{mainsph}
 Let $0<\lambda<Q=2n+2$ and $p:=2Q/(2Q-\lambda)$. Then for any $f,g\in L^p(\Sph^{2n+1})$
\begin{equation}
 \label{eq:mainsph}
\left| \iint_{\Sph^{2n+1}\times\Sph^{2n+1}} \frac{\overline{f(\zeta)}\ g(\eta)}{|1-\zeta\cdot\overline\eta|^{\lambda/2}} \,d\zeta\,d\eta \right|
\leq \left(\frac{2\pi^{n+1}}{n!} \right)^{\lambda/ Q} \frac{n!\, \Gamma((Q-\lambda)/2)}{\Gamma^2((2Q-\lambda)/4)}
\ \|f\|_p\ \|g\|_p
\end{equation}
with equality if and only if
\begin{equation}
 \label{eq:optsph}
f(\zeta) = \frac{c}{|1-\overline\xi\cdot\zeta|^{(2Q-\lambda)/2}} \,,
\qquad
g(\zeta) = \frac{c'}{|1-\overline\xi\cdot\zeta|^{(2Q-\lambda)/2}} \,,
\end{equation}
for some $c,c'\in\C$ and some $\xi\in\C^{n+1}$ with $|\xi|<1$ (unless $f\equiv 0$ or $g\equiv 0$).
\end{theorem}

In particular, $f=g\equiv1$ are optimizers and this enables us to compute the constant.

\subsection*{Sobolev inequalities on the sphere}

Just as on $\Hei^n$ there is a duality between the fractional integral inequality \eqref{eq:mainsph} with $\lambda=Q-2$ and a Sobolev inequality on the sphere $\Sph^{2n+1}$. In order to state this inequality, we first need to introduce some notation. For $j=1,\ldots,n+1$ we define the operators
$$
T_j := \frac{\partial}{\partial \zeta_j}-\overline{\zeta_j} \sum_{k=1}^{n+1} \zeta_k \frac{\partial}{\partial \zeta_k}  \,,
\qquad
\overline{T_j} := \frac{\partial}{\partial \overline{\zeta_j}}-\zeta_j \sum_{k=1}^{n+1} \overline{\zeta_k} \frac{\partial}{\partial \overline{\zeta_k}} \,,
$$
and the conformal Laplacian
$$
\mathcal L := -\frac12 \sum_{j=1}^{n+1} \left(\overline{T_j}T_j + T_j \overline{T_j}\right) + \frac{n^2}4 \,.
$$
The associated quadratic form is
\begin{equation}
 \label{eq:esph}
\mathcal E[u] := \frac12 \int_{\Sph^{2n+1}} \left( \sum_{j=1}^{n+1} \left( |T_j u|^2 + |\overline{T_j} u|^2 \right) + \frac{n^2}2 |u|^2 \right) \,d\zeta \,.
\end{equation}
The Sobolev (or Folland-Stein) space $S^1(\Sph^{2n+1})$ consists of all functions $u$ on $\Sph^{2n+1}$ satisfying $\mathcal E[u]<\infty$. With this notation Theorem \ref{mainsph} with $\lambda=Q-2$ is equivalent to the Jerison-Lee inequality
\begin{equation}
 \label{eq:jlsph}
\mathcal E[u] \geq \frac{n^2}{4} \left(\frac{2\pi^{n+1}}{n!}\right)^{2/Q}  \left( \int_{\Sph^{2n+1}} |u|^{2Q/(Q-2)} \,d\zeta \right)^{(Q-2)/Q}
\end{equation}
for all $u\in S^1(\Sph^{2n+1})$. We will discuss this (sharp) inequality and the cases of equality again in the following Section \ref{sec:jl}. There are more inequalities that one can deduce from \eqref{eq:mainsph}. The following is new. Let $\mathcal E_0[u]$ be given by \eqref{eq:esph} without the term $\frac{n^2}2 |u|^2$.

\begin{corollary}\label{sobq}
 Let $2< q< 2Q/(Q-2)$. Then for any $u\in S^1(\Sph^{2n+1})$
\begin{equation}
 \label{eq:sobq}
\frac{4(q-2)}{Q-2} \mathcal E_0[u] + \int_{\Sph^{2n+1}} |u|^2 \,d\zeta 
\geq |\Sph^{2n+1}|^{1-2/q} \left( \int_{\Sph^{2n+1}} |u|^q \,d\zeta \right)^{2/q} \,.
\end{equation}
Equality holds if and only if $u$ is constant.
\end{corollary}

This corollary is the analogue of a Sobolev inequality of \cite{BiVe,Be2} for functions on the Riemannian sphere. Equation \eqref{eq:sobq} agrees with \eqref{eq:jlsph} if $q=2Q/(Q-2)$, but we state \eqref{eq:jlsph} separately because the family of optimizers is different in the two cases. The proof of Corollary~\ref{sobq} will be given in Subsection \ref{sec:sobq}. It is related to arguments in \cite{Be2}.

\subsection*{The limiting cases}
We conclude this section by presenting two inequalities that follow via differentiation at the endpoints $\lambda=0$ and $\lambda=Q$. We only state them for functions on the sphere, but there are equivalent versions on the Heisenberg group obtained via the Cayley transform. Our first corollary is, in fact, the main result of \cite{BrFoMo}. It is the $\Hei^n$ version of \cite{CaLo2,Be2}.

\begin{corollary}\label{mainlog}
 For any non-negative $f,g\in L\log L(\Sph^{2n+1})$ with
$$
\int_{\Sph^{2n+1}} f\,d\zeta = \int_{\Sph^{2n+1}} g\,d\zeta = |\Sph^{2n+1}| = \frac{2\pi^{n+1}}{n!}
$$
one has
\begin{equation}
 \label{eq:mainlog}
\begin{split}
 & \iint_{\Sph^{2n+1}\times\Sph^{2n+1}} f(\zeta) \log\left(\frac{1}{|1-\zeta\cdot\overline\eta|}\right) g(\eta) \,d\zeta\,d\eta \\
 & \qquad \qquad \leq \frac{|\Sph^{2n+1}|}{Q} \int_{\Sph^{2n+1}} f\log f\,d\zeta 
+ \frac{|\Sph^{2n+1}|}{Q} \int_{\Sph^{2n+1}} g\log g\,d\zeta \,.
\end{split}
\end{equation}
The constant $|\Sph^{2n+1}|/Q$ is sharp and equality holds if $f$ and $g$ are $L^1$-normalized functions given in \eqref{eq:optsph} with $\lambda=0$.
\end{corollary}

It is shown in \cite{BrFoMo} that the stated functions are the \emph{only} optimizers.

The next corollary, corresponding to the endpoint $\lambda=Q$, is new.

\begin{corollary}\label{mainent}
 For any non-negative $f\in L^2\log L(\Sph^{2n+1})$ with
$$
\int_{\Sph^{2n+1}} f^2 \,d\zeta = |\Sph^{2n+1}| = \frac{2\pi^{n+1}}{n!}
$$
one has
\begin{equation}
 \label{eq:mainent}
 \iint_{\Sph^{2n+1}\times\Sph^{2n+1}} \frac{|f(\zeta)-f(\eta)|^2}{|1-\zeta\cdot\overline\eta|^{Q/2}} \,d\zeta\,d\eta
\geq \frac{2\pi^{n+1}}{\Gamma(Q/4) \, \Gamma((Q+4)/4)} \int_{\Sph^{2n+1}} f^2\log f^2 \,d\zeta \,.
\end{equation}
The constant $2\pi^{n+1}/(\Gamma(Q/4) \, \Gamma((Q+4)/4))$ is sharp and equality holds for the $L^2$-normalized functions $f$ given in \eqref{eq:optsph} with $\lambda=Q$.
\end{corollary}

Theorems \ref{main} and \ref{mainsph}, as well as Corollaries \ref{mainlog} and \ref{mainent} are proved in Section \ref{sec:proofs}.


\section{The inequality of Jerison and Lee}\label{sec:jl}

As we said in the introduction, the first example \cite{JeLe} of a sharp constant for the Heisenberg group was inequality \eqref{eq:sobjl}. In this section we rederive their result by our methods which, in the $\lambda=Q-2$ case, we believe to be simpler than both the method in \cite{JeLe} and the general $\lambda$ case in the rest of the paper. We do so also to expose the strategy of our proof most clearly. It is easiest for us to work in the formulation on the sphere $\Sph^{2n+1}$, and we do so. Recall that $\mathcal E[u]$ is defined in \eqref{eq:esph}.

\begin{theorem}\label{mainjl}
For all $u\in S^1(\Sph^{2n+1})$ one has
\begin{equation}\label{eq:mainjl}
\mathcal E[u] \geq \frac{n^2}{4} \left(\frac{2\pi^{n+1}}{n!}\right)^{2/Q}  \left( \int_{\Sph^{2n+1}} |u|^{2Q/(Q-2)} \,d\zeta \right)^{(Q-2)/Q} \,,
\end{equation}
with equality if and only if
\begin{equation}
 \label{eq:optjl}
u(\zeta) = \frac{c}{|1-\overline{\xi}\cdot\zeta|^{(Q-2)/2}}
\end{equation}
for some $c\in\C$ and some $\xi\in\C^{n+1}$ with $|\xi|<1$.
\end{theorem}

See Appendix \ref{sec:equiv} for the equivalence of the $\Hei^n$-version \eqref{eq:sobjl} and the $\Sph^{2n+1}$-version \eqref{eq:mainjl} of the Sobolev inequality. Both are in \cite{JeLe}.

By a duality argument (cf. \cite[Thm. 8.3]{LiLo} for a Euclidean version) based on the fact \cite{FoSt} that $|u^{-1}v|^{-Q+2}$ is a constant times the Green's function of the sub-Laplacian on $\Hei^n$, this theorem is equivalent to the case $\lambda=Q-2$ of Theorem \ref{mainsph}.

We shall make use of the following elementary formula.

\begin{lemma}\label{gsr}
 For any real-valued $u\in S^1(\Sph^{2n+1})$ one has
\begin{equation}
\label{eq:gsr}
\sum_{j=1}^{n+1} \mathcal E[\zeta_j u] = \mathcal E[u] + \frac{n}{2} \int u^2 \,d\zeta \,.
\end{equation}
\end{lemma}

\begin{proof}
We begin by noting that for any smooth function $\phi$ on $\Sph^{2n+1}$ one has
\begin{align*}
|T_k (\phi u)|^2 + |\overline{T_k} (\phi u)|^2 = & |\phi|^2 \left( |T_k u|^2 + |\overline{T_k} u|^2 \right)
+ u^2 \left( |T_k \phi|^2 + |\overline{T_k} \phi|^2 \right) \\
& + \re \left( T_k(u^2) \overline{\phi} \overline{T_k}\phi + \overline{T_k}(u^2) \overline{\phi} T_k\phi \right) \,.
\end{align*}
We integrate this identity over $\Sph^{2n+1}$ and use the fact that the $L^2$-adjoint of $T_k$ satisfies $T_k^* = -\overline{T_k}$ in order to obtain
\begin{align*}
\int \left( |T_k (\phi u)|^2 + |\overline{T_k} (\phi u)|^2 \right) \,d\zeta & = \int \left( |\phi|^2 \left( |T_k u|^2 + |\overline{T_k} u|^2 \right)
+ u^2 \left( |T_k \phi|^2 + |\overline{T_k} \phi|^2 \right) \right) \,d\zeta \\
& \quad - \re \int u^2 \left( T_k(\overline{\phi} \overline{T_k}\phi) + \overline{T_k}(\overline{\phi} T_k\phi) \right) \,d\zeta \\
& = \int |\phi|^2 \left( |T_k u|^2 + |\overline{T_k} u|^2 \right) \,d\zeta \\
& \quad - \re \int u^2 \overline{\phi} \left( T_k \overline{T_k}\phi + \overline{T_k} T_k\phi) \right) \,d\zeta \,.
\end{align*}
Summing over $k$ we find that
$$
\mathcal E[\phi u] = \frac12 \sum_{k=1}^{n+1} \int |\phi|^2 \left( |T_k u|^2 + |\overline{T_k} u|^2 \right) \,d\zeta
+ \re \int u^2 \overline{\phi} \mathcal L\phi \,d\zeta \,.
$$
We apply this identity to $\phi(\zeta)=\zeta_j$. Using that
$$
T_k \zeta_j = \delta_{j,k} - \overline{\zeta_k} \zeta_j \,,
\qquad
\overline{T_k} \zeta_j = 0 \,,
$$
we find
$$
\mathcal L \zeta_j = \frac{n}{2}\left(\frac{n}{2}+1 \right) \zeta_j \,,
$$
and therefore
$$
\mathcal E[\zeta_j u] = \frac12 \sum_{k=1}^{n+1} \int |\zeta_j|^2 \left( |T_k u|^2 + |\overline{T_k} u|^2 \right) \,d\zeta
+ \frac{n}{2}\left(\frac{n}{2}+1 \right) \int |\zeta_j|^2 u^2 \,d\zeta \,.
$$
Summing over $j$ yields \eqref{eq:gsr} and completes the proof.
\end{proof}

We are now ready to give a short

\begin{proof}[Proof of Theorem \ref{mainjl}]
We know from \cite{JeLe1} that there is an optimizer $w$ for inequality \eqref{eq:mainjl}. (Using the Cayley transform, one can deduce this also from our Proposition~\ref{exmax} and the known form of the fundamental solution of $\mathcal L$.)

As a preliminary remark we note that any optimizer is a complex multiple of a non-negative function. Indeed, if $u=a+ib$ with $a$ and $b$ real functions, then $\mathcal E[u]=\mathcal E[a]+\mathcal E[b]$. We also note that the right side of \eqref{eq:mainjl} is $\|a^2+b^2\|_{q/2}$ with $q=2Q/(Q-2)>2$. By the triangle inequality, $\|a^2+b^2\|_{q/2}\leq \|a^2\|_{q/2}+\|b^2\|_{q/2}$. This inequality is strict unless $a\equiv 0$ or $b^2=\lambda^2 a^2$ for some $\lambda\geq 0$. Therefore, if $w=a+ib$ is an optimizer for \eqref{eq:mainjl}, then either one of $a$ and $b$ is identically equal to zero or else both $a$ and $b$ are optimizers and $|b|=\lambda |a|$ for some $\lambda>0$. For any real $u\in S^1(\Sph^{2n+1})$ its positive and negative parts $u_\pm$ belong to $S^1(\Sph^{2n+1})$ and satisfy $\partial u_\pm/\partial\zeta_k = \pm\chi_{\{\pm u>0\}} \partial u/\partial\zeta_k$ in the sense of distributions. (This can be proved similarly to \cite[Thm. 6.17]{LiLo}.) Thus $\mathcal E[u]=\mathcal E[u_+]+\mathcal E[u_-]$ for real $u$. Moreover, $\| u\|_q^2 \leq \|u_+\|_q^2 + \|u_-\|_q^2$ for real $u$ with strict inequality unless $u$ has a definite sign. Therefore, if $w=a+ib$ is an optimizer for \eqref{eq:mainjl}, then both $a$ and $b$ have a definite sign. We conclude that any optimizer is a complex multiple of a non-negative function. Hence we may assume that $w\geq 0$.

It is important for us to know that we may confine our search for optimizers to functions $u$ satisfying the \emph{center of mass condition}
\begin{equation}\label{eq:comjl}
\int_{\Sph^{2n+1}} \zeta_j\ |u(\zeta)|^q \,d\zeta= 0\,,
\qquad j=1,\ldots ,n+1 \,.
\end{equation}
It is well-known, and used in many papers on this subject (e.g., \cite{He,On,ChYa,BrFoMo}), that this can be assumed, and we give a proof of this fact in Appendix \ref{sec:com}. It uses three facts: one is that inequality \eqref{eq:mainjl} is invariant under $U(n+1)$ rotations of $\Sph^{2n+1}$. The second is that the Cayley transform, that maps $\Hei^n$ to $\Sph^{2n+1}$, leaves the optimization problem invariant. The third is that the $\Hei^n$-version, \eqref{eq:sobjl}, of inequality \eqref{eq:mainjl} is invariant under dilations $F(u) \mapsto \delta^{(Q-2)/2} F(\delta u)$. Our claim in the appendix is that by a suitable choice of $\delta$ and a rotation we can achieve \eqref{eq:comjl}.

We may assume, therefore, that the optimizer $w$ satisfies \eqref{eq:comjl}. Imposing this constraint does not change the positivity of $w$. We shall prove that the only optimizer with this property is the constant function. It follows, then, that the only optimizers without condition \eqref{eq:comjl} are those functions for which the dilation and rotation, just mentioned, yields a constant. In Appendix \ref{sec:com} we identify those functions as the functions stated in \eqref{eq:optjl}.

The second variation of the quotient $\mathcal E[u]/\|u\|_q^2$ around $u=w$ shows that
\begin{equation}
 \label{eq:secvarjl}
\mathcal E[v] \int_{\Sph^{2n+1}} w^q \,d\zeta - (q-1) \mathcal E[w] \int_{\Sph^{2n+1}} w^{q-2} |v|^2 \,d\zeta \geq 0
\end{equation}
for all $v$ with $\int w^{q-1} v \,d\zeta =0$. Inequality \eqref{eq:secvarjl} is proved by first considering \emph{real} variations, in which case it is straightforward, and then handling complex changes $v=a+ib$ by adding the inequalities for $a$ and $b$ and using that $\mathcal E[v]=\mathcal E[a]+\mathcal E[b]$, as noted above.

Because $w$ satisfies condition \eqref{eq:comjl} we may choose $v(\zeta)=\zeta_j w(\zeta)$ in \eqref{eq:secvarjl} and sum over $j$. We find
\begin{equation}\label{eq:secondvarjl}
\sum_{j=1}^{n+1} \mathcal E[\zeta_j w] \geq (q-1) \mathcal E[w] \,.
\end{equation}
On the other hand, Lemma \ref{gsr} with $u=w$ implies
\begin{equation*}
\sum_{j=1}^{n+1} \mathcal E[\zeta_j w] = \mathcal E[w] + \frac{n}{2} \int w^2 \,d\zeta \,,
\end{equation*}
which, together with \eqref{eq:secondvarjl}, yields
$$
\frac{n}{2} \int w^2 \,d\zeta \geq (q-2) \mathcal E[w] \,.
$$
Recalling that $q-2=2/n$, we see that this is the same as
$$
\sum_{j=1}^{n+1} \int \left( |T_j w|^2 + |\overline{T_j} w|^2 \right) \,d\zeta = 0 \,.
$$
Since the operator $\mathcal L-n^2/4$ is positive definite on the orthogonal complement of constants we conclude that $w$ is the constant function, as we intended to prove.
\end{proof}


\section{Existence of an optimizer}

Our goal in this section will be to show that the optimization problem corresponding to inequality \eqref{eq:main} admits an optimizer for all $\lambda$. Our proof relies on the fact that convolution with $|u|^{-\lambda}$ is a positive definite operator. In contrast to the Euclidean case, this property is not completely obvious in the setting of the Heisenberg group and we shall prove it in Subsection \ref{sec:sr}.

This positive definiteness together with a duality argument allows us to reformulate \eqref{eq:main} as a maximization problem with an $L^2$ constraint instead of the $L^p$, $p\neq 2$, constraint appearing in \eqref{eq:main}. We shall prove the existence of an optimizer of this equivalent problem in Subsection \ref{sec:exmax}.

We denote the (non-commutative) convolution on the Heisenberg group by
$$
f*g(u) = \int_{\Hei^n} f(v) g(v^{-1}u) \,dv \,.
$$
Moreover, we introduce the \emph{sublaplacian}
\begin{equation}
 \label{eq:sublap}
\mathcal L := - \frac14 \sum_{j=1}^n \left( \left( \frac\partial{\partial x_j} +2 y_j \frac\partial{\partial t} \right)^2 + \left( \frac\partial{\partial y_j} -2 x_j \frac\partial{\partial t} \right)^2 \right) \,.
\end{equation}
Here, we use the same notation $\mathcal L$ as for the conformal Laplacian on the sphere, but it will be clear from the context which operator is meant.


\subsection{The operator square root of convolution with $|u|^{-\lambda}$}
\label{sec:sr}

Although it is not obvious, the operator of convolution with the function $|u|^{-\lambda}$ on $\Hei^n$ is positive definite and its operator square root is again a convolution operator. In the Euclidean case, in contrast, the formula
$$
|x-y|^{-\lambda} = \const \int_{\R^N} |x-z|^{-(N+\lambda)/2} |y-z|^{-(N+\lambda)/2} \,dz
$$
shows that convolution with $|x|^{-\lambda}$ is positive definite and, at the same time, provides a formula for its square root. The analogous guess for the Heisenberg group, namely convolution with $|u|^{-(Q+\lambda)/2}$, is, unfortunately, \emph{not} the square root of convolution with $|u|^{-\lambda}$, although it is dimensionally right and it is close to the correct answer. Positive definiteness of $|u|^{-\lambda}$ was shown by \cite{Co} by explicitly computing its eigenvalues. This computation provides a spectral representation for the kernel of the square root as well. There does not seem to be a simple, closed-form expression for this square root as there is for the Euclidean case, and some work is needed to elucidate its properties. In our proof we utilize our `almost correct guess' together with a recent multiplier theorem by M\"uller, Ricci and Stein \cite{MuRiSt}.

\begin{proposition}
 \label{sr}
Let $0<\lambda<Q$. There is a function $k\in L^{2Q/(Q+\lambda)}_{\mathrm w}(\Hei^n)$ such that
\begin{equation}
 \label{eq:sr}
|u^{-1} v|^{-\lambda} = \int_{\Hei^n} k(u^{-1} w) k(v^{-1} w) \,dw
\qquad\text{for all}\ u,v\in\Hei^n \,.
\end{equation}
The function $k$ is real-valued, even and homogeneous of degree $-(Q+\lambda)/2$.
\end{proposition}

Here `even' means $k(u^{-1})=k(u)$ for all $u\in\Hei^n$ and `homogeneous of degree $\alpha$' means $k(\delta u)=\delta^\alpha k(u)$ for all $u\in\Hei^n$ and all $\delta>0$.

\begin{proof}
 Besides the sublaplacian \eqref{eq:sublap} we shall use the operator $\mathcal T=\frac\partial{\partial t}$. These two operators commute. It was shown by Cowling \cite{Co} (see also \cite[Sec.1]{BrFoMo}) that for $0<s<Q/2$ the function $|u|^{-Q+2s}$ is a constant times the fundamental solution of the operator
\begin{equation}
 \label{eq:ls}
\mathcal L_s := |2\mathcal T|^{s} \, \frac{\Gamma(|2\mathcal T|^{-1}\mathcal L+\tfrac{1+s}{2})}{\Gamma(|2\mathcal T|^{-1}\mathcal L+\tfrac{1-s}{2})} \,.
\end{equation}
More precisely,
\begin{equation}
 \label{eq:fundsol}
(\mathcal L_s^{-1} \delta_0)(u) = a_s |u|^{-Q+2s}\,,
\qquad a_s = \frac{2^{n-s-1} \Gamma^2(\tfrac{Q-2s}4)}{\pi^{n+1} \Gamma(s)} \,,
\end{equation}
where $\delta_0$ denotes a Dirac delta at the point $0$. We note that $\mathcal L_1=\mathcal L$, for which the fundamental solution has been computed in \cite{FoSt}.

For given $\lambda$, we abbreviate $s:=(Q-\lambda)/2$ and define
\begin{equation}
 \label{eq:defk}
k:= a_s^{-1/2}\ \mathcal L_s^{-1/2} \delta_0 \,.
\end{equation}
Since $\mathcal L_s^{-1/2}\mathcal L_s^{-1/2} = \mathcal L_s^{-1}$, this function satisfies
$$
\int_{\Hei^n} k(w^{-1}u) k(v^{-1}w) \,dw = |u^{-1} v|^{-\lambda} \,,
$$
which, modulo the fact that $k$ is even, coincides with \eqref{eq:sr}.

We have to show that the formal definition \eqref{eq:defk} actually defines a function as stated in the proposition.
The key to obtaining these properties is the representation
\begin{equation}
 \label{eq:reprsr}
k= a_s^{-1/2} a_{s/2}\ m(|2\mathcal T|^{-1}\mathcal L) |u|^{-(Q+\lambda)/2}
\end{equation}
with
\begin{equation}
\label{eq:mult}
m(E) := \sqrt{\frac{\Gamma(E+\tfrac{1-s}{2})}{\Gamma(E+\tfrac{1+s}{2})}}
\ \frac{\Gamma(E+\tfrac{2+s}{4})}{\Gamma(E+\tfrac{2-s}{4})} \,.
\end{equation}
Relation \eqref{eq:reprsr} follows from
$$
k= a_s^{-1/2}\ \mathcal L_s^{-1/2} \mathcal L_{s/2} \mathcal L_{s/2}^{-1} \delta_0 = a_s^{-1/2} a_{s/2}\ \mathcal L_s^{-1/2} \mathcal L_{s/2} |u|^{-(Q+\lambda)/2} \,,
$$
where we used \eqref{eq:fundsol}, together with the fact that $\mathcal L_s^{-1/2} \mathcal L_{s/2} = m(|2\mathcal T|^{-1}\mathcal L)$, which follows from \eqref{eq:ls}.

Since the function $|u|^{-(Q+\lambda)/2}$ appearing in \eqref{eq:reprsr} has all the properties stated in the proposition, it remains to check that these are preserved under the operator $m(|2T|^{-1}\mathcal L)$. The operator $|2\mathcal T|^{-1}\mathcal L$, and hence also $m(|2\mathcal T|^{-1}\mathcal L)$, commutes with inversion $u\mapsto u^{-1}$ and with scalings $u\mapsto \delta u$. Since $|u|^{-(Q+\lambda)/2}$ is even and homogeneous of degree $-(Q+\lambda)/2$, the same is true for $k$. Since convolution with $k$ is self-adjoint, the fact that $k$ is even implies that it is real-valued. Moreover, $|u|^{-(Q+\lambda)/2}\in L^{2Q/(Q+\lambda)}_{\mathrm w}$, so in order to deduce the same property for $k$ it suffices to show that $m(|2\mathcal T|^{-1}\mathcal L)$ maps $L^{2Q/(Q+\lambda)}_{\mathrm w}$ into itself. By the Marcinkiewicz interpolation theorem (as extended in \cite[Thm. 3.15]{StWe}) it suffices to show that it maps $L^p$ into itself for all $1<p<\infty$. This, in turn, follows from the multiplier theorem in \cite{MuRiSt} if we can show that $\left(E \frac{d}{dE} \right)^\nu m(E)$ is bounded on $[n/2,\infty)$ for any $\nu\in\N_0$. In fact, we will prove that
$$
\left| \left(E \frac{d}{dE} \right)^\nu \log m(E)\right| \leq C_\nu
\qquad \text{for all}\ E\in [n/2,\infty) \,.
$$
Note that this is only a problem for large $E$. We write
$$
\log m(E) = - \frac 12 \int_{E+\tfrac{1-s}{2}}^{E+\tfrac{1+s}{2}} \psi(t) \,dt
+ \int_{E+\tfrac{2-s}{4}}^{E+\tfrac{2+s}{4}} \psi(t) \,dt
= \int \chi(t-E-\tfrac12) \, \psi(t) \,dt \,,
$$
where $\psi:=(\log\Gamma)'$ denotes the Digamma function and where $\chi(t):=\frac12$ if $|t|\leq \frac s4$, $\chi(t):=-\frac12$ if $\frac s4<|t|\leq \frac s2$ and $\chi(t):=0$ otherwise. The assertion now follows from the integral representation
$$
\psi(t) = \log t-\frac1{2t}-2 \int_0^\infty \frac{\tau\,d\tau}{(\tau^2+t^2)(e^{2\pi\tau}-1)}
$$
(see \cite[(6.3.21)]{AbSt}) and some elementary calculations.
\end{proof}

For later reference we mention a bound with a similar, but much simpler proof.

\begin{lemma}\label{equiv}
 Let $0<\lambda<Q$, $s:=(Q-\lambda)/2$ and let $k$ be the function in Proposition~\ref{sr}. Then there is a constant $C$ such that for all $f\in L^2(\Hei^n)$ one has
$$
\|\mathcal L^{s/2} (f*k) \|_2 \leq C \|f\|_2 \,.
$$
\end{lemma}

\begin{proof}
 Since $\mathcal L_s^{1/2} (f*k)= a_s^{-1/2} f$ in the notation of the proof of Proposition \ref{sr}, we have to prove that the operator
$$
\mathcal L^{s/2} \mathcal L_s^{-1/2} = \mathcal L^{s/2} |2\mathcal T|^{-s/2}  \sqrt{\frac{\Gamma(|2\mathcal T|^{-1}\mathcal L+\tfrac{1-s}{2})}{\Gamma(|2\mathcal T|^{-1}\mathcal L+\tfrac{1+s}{2})} }
$$
is bounded in $L^2(\Hei^n)$. Since $\mathcal L$ and $\mathcal T$ commute this follows immediately from the boundedness of the function
$\tilde m(E) = \sqrt{E^s \, \Gamma(E+\tfrac{1-s}{2})/\Gamma(E+\tfrac{1+s}{2}) }$ on $[n/2,\infty)$, which is easily checked using Stirling's formula.
\end{proof}


\subsection{Existence of an optimizer}\label{sec:exmax}

In this subsection we consider the optimization problem
\begin{equation}
 \label{eq:exmax}
\mathcal C_{n,\lambda} := \sup\{ \|f*k\|_q :\ \|f\|_2 =1 \} \,,
\end{equation}
where $k$ is the function in Proposition~\ref{sr} for a fixed $0<\lambda<Q$ and where $q:= 2Q/\lambda$. As we shall explain now, this optimization problem is equivalent to the one corresponding to inequality \eqref{eq:main}. For assume that we can prove that the inequality
\begin{equation}
 \label{eq:dualineq}
\|f*k\|_q \leq \mathcal C_{n,\lambda} \|f\|_2
\end{equation}
has an optimizer. Since $k$ is real and even, the mapping $f\mapsto f*k$ is self-adjoint in $L^2(\Hei^n)$. Hence, by duality, we infer that the inequality
\begin{equation}
 \label{eq:preineq}
\|f*k\|_2 \leq \mathcal C_{n,\lambda} \|f\|_p \,,
\qquad p=q'=2Q/(2Q-\lambda) \,,
\end{equation}
has an optimizer. But the latter inequality is the same as
$$
\iint_{\Hei^n\times\Hei^n} \overline{f(u)} l(u,v) f(v) \,du\,dv \leq \mathcal C_{n,\lambda}^2 \|f\|_p^2
$$
with
$$
l(u,v) := \int_{\Hei^n} k(u^{-1}w) k(v^{-1} w) \,dw = |u^{-1}v|^{-\lambda} \,,
$$
according to Proposition \ref{sr}, and this is inequality \eqref{eq:main} with $f=g$. Since the kernel $|u^{-1}v|^{-\lambda}$ is positive definite, the case $f=g$ is the only one that needs to be considered.

We shall now prove the existence of an optimizer for \eqref{eq:dualineq}.

\begin{proposition}[\textbf{Existence of an optimizer}]\label{exmax}
Let $0<\lambda<Q$, $q:= 2Q/\lambda$ and let $k$ be the function in Proposition~\ref{sr}. Then the supremum \eqref{eq:exmax} is attained. Moreover, for any maximizing sequence $(f_j)$ there is a subsequence $(f_{j_m})$ and sequences $(a_m)\subset\Hei^n$ and $(\delta_m)\subset(0,\infty)$ such that
$$
g_m(u) := \delta_m^{Q/2} f_{j_m}(\delta_m(a_m^{-1}u))
$$
converges strongly in $L^2(\Hei^n)$.
\end{proposition}

Of course, the optimization problem \eqref{eq:exmax} is translation and dilation invariant, which leads to loss of compactness in two ways. What we shall prove is that these are the \emph{only} ways; in other words, after translating back by $(a_m)$ and dilating back by $(\delta_m)$ the maximizing sequence has a strongly convergent subsequence.

Our proof of Proposition \ref{exmax} simplifies and extends proofs in \cite{Ge,KiVi} for the case of the Euclidean Sobolev inequality. It is based on two ingredients. The first one is an improvement of inequality \eqref{eq:dualineq} in terms of a Besov norm, which we quote from \cite{BaGeXu} and \cite{BaGa}. Its statement involves the semi-group $e^{-t\mathcal L}$ of the sub-Laplacian $\mathcal L$ (see \eqref{eq:sublap}) defined as a self-adjoint, non-negative operator in $L^2(\Hei^n)$. The operators $e^{-t\mathcal L}$ are defined by the spectral theorem and extended by continuity to $L^q(\Hei^n)$. We have

\begin{lemma}[\textbf{Refined HLS inequality}]\label{besov}
 Let $0<\lambda<Q$, $q:= 2Q/\lambda$ and let $k$ be the function in Proposition~\ref{sr}. Then there is a constant $c_{\lambda,n}$ such that for any $f\in L^2(\Hei^n)$
$$
\|f*k\|_q \leq c_{\lambda,n} \|f\|_2^{\lambda/Q} \left( \sup_{\beta>0} \beta^{\lambda/4} \| e^{-\beta \mathcal L} (f*k) \|_\infty \right)^{(Q-\lambda)/Q} \,.
$$
\end{lemma}

To be more precise, the paper \cite{BaGa} contains the inequality
$$
\|\psi\|_q \leq \tilde c_{s,n} \|\mathcal L^{s/2} \psi \|_2^{(Q-2s)/Q} \left( \sup_{\beta>0} \beta^{(Q-2s)/4} \| e^{-\beta \mathcal L} \psi \|_\infty \right)^{2s/Q}
$$
for $0<s<Q/2$. One obtains Lemma \ref{besov} by applying this bound with $s=(Q-\lambda)/2$ to the function $\psi=f*k$ and using Lemma \ref{equiv}.

The second ingredient in our proof of Proposition \ref{exmax} is the following Rellich-Kondrashov-type lemma.

\begin{lemma}[\textbf{a.e. convergence}]\label{rellich}
 Let $0<\lambda<Q$ and let $k$ be the function in Pro\-po\-si\-tion~\ref{sr}. If $(f_j)$ is a bounded sequence in $L^2(\Hei^n)$ then a subsequence of $(f_j*k)$ converges a.e. and in $L^r_{\loc}$ for all $r<2Q/\lambda$.
\end{lemma}

\begin{proof}[Proof of Lemma \ref{rellich}]
We will need to replace $k(u)$ by $\tilde k(u):= |u|^{-(Q+\lambda)/2}$. (The reason for this is that Proposition \ref{sr} does not guarantee that $k$ is square-integrable on a sphere; if it is, then the homogeneity will guarantee that $k$ is square integrable at infinity.) We define $g_j:= m(|2\mathcal T|^{-1}\mathcal L) f_j$ with the multiplier $m(|2\mathcal T|^{-1}\mathcal L)$ given by \eqref{eq:mult}. As we have seen in the proof of Proposition \ref{sr}, this is a bounded operator in $L^2$, and hence $(g_j)$ is a bounded sequence in $L^2$. Below we shall prove that a subsequence of $(g_j*\tilde k)$ converges a.e. and in $L^r_{\loc}$ for all $r<2Q/\lambda$. Since $f_j*k=a_s^{-1/2} a_{s/2}\ g_j*\tilde k$ by \eqref{eq:reprsr}, this will prove the assertion of the lemma.

Since $g_j$ is bounded in $L^2(\Hei^n)$ we can (modulo passing to a subsequence) assume that it converges weakly to some $g$. We shall prove that for any set $\Omega\subset\Hei^n$ of finite measure and any $r<2Q/\lambda$, $(g_j*\tilde k)$ converges to $g*\tilde k$ in $L^r(\Omega)$. This implies, as is well known, that a subsequence of $(g_j*\tilde k)$ converges to $g*\tilde k$ a.e. on $\Omega$, and the claim then follows by a diagonal argument.

In order to prove the claimed convergence in $L^r(\Omega)$, we decompose $\tilde k=\tilde l+\tilde m$, where $\tilde l(u)=\tilde k(u)\chi_{\{|u|>\rho\}}$ and $\tilde m(u)=\tilde k(u)\chi_{\{|u|<\rho\}}$, and where $\rho>0$ is a parameter to be chosen later. Since, for any fixed $u$, the function $\tilde l(v^{-1}u)$ is square integrable with respect to $v$, weak convergence implies that $g_j*\tilde l \to g*\tilde l$ pointwise. Moreover, $|(g_j*\tilde l)(u)| \leq \|g_j\|_2 \|\tilde l\|_2\leq C(\rho)$, independent of $u$ and $j$. By dominated convergence, this implies that $g_j*\tilde l\to g*\tilde l$ in $L^r(\Omega)$.

In order to control $g_j*\tilde m$, we let $s:= 2r/(r+2)$ and note that $s< 2Q/(Q+\lambda)=:\sigma$. Hence $\tilde m\in L^s(\Hei^n)$ and $\|\tilde m\|_s=\const \rho^\alpha$ with $\alpha=Q(1/s-1/\sigma)>0$. Hence by Young's inequality on $\Hei^n$ and the boundedness of $(g_j)$ we find that $\|(g_j-g)*\tilde m\|_r \leq \|g_j-g\|_2 \|\tilde m\|_s \leq \const \rho^\alpha$. Choosing first $\rho$ small and then $j$ large we verify the claimed convergence in $L^r$.
\end{proof}

The following consequence of Lemmas \ref{besov} and \ref{rellich} is the crucial ingredient to prove the existence of an optimizer in \eqref{eq:exmax}.

\begin{corollary}\label{concomp}
 Let $0<\lambda<Q$, $q:= 2Q/\lambda$ and let $k$ be the function in Proposition~\ref{sr}. Let $(f_j)$ be a bounded sequence in $L^2(\Hei^n)$. Then one of the following alternatives occurs.
\begin{enumerate}
 \item $(f_j*k)$ converges to zero in $L^q(\Hei^n)$.
\item There is a subsequence $(f_{j_m})$ and sequences $(a_m)\subset\Hei^n$ and $(\delta_m)\subset(0,\infty)$ such that
$$
g_m(u) := \delta_m^{Q/2} f_{j_m}(\delta_m(a_m^{-1}u))
$$
converges weakly in $L^2(\Hei^n)$ to a function $g\not\equiv 0$. Moreover, $(g_m*k)$ converges a.e. and in $L^r_{\loc}(\Hei^n)$, $r<q$, to $g*k$.
\end{enumerate}
\end{corollary}

\begin{proof}
 Let $(f_j)$ be bounded in $L^2(\Hei^n)$ and assume that $(f_j*k)$ does not converges to zero in $L^q(\Hei^n)$. Then, after passing to a subsequence, we may assume that $\|f_j*k\|_q \geq \epsilon$ for some $\epsilon>0$ and all $j$. Since $(f_j)$ is bounded in $L^2(\Hei^n)$, Lemma \ref{besov} yields
$$
\sup_{\beta>0} \beta^{\lambda/4} \| e^{-\beta \mathcal L} (f_j*k) \|_\infty \geq \tilde\epsilon \,, 
$$
that is, there are $\beta_j>0$ and $u_j\in\Hei^n$ such that
$$
\left| \beta_j^{\lambda/4} \left( e^{-\beta_j \mathcal L} (f_j*k) \right)(u_j) \right| \geq \tilde\epsilon \,.
$$
Next, we use the fact that $e^{-\beta \mathcal L}$ is a convolution operator. More precisely, there is a smooth, rapidly decaying function $G$ on $\Hei^n$ such that
$$
e^{-\beta \mathcal L} f = \beta^{-Q/2} f * G(\beta^{-1/2} \cdot) \,,
$$
see, e.g., \cite{Ga,Hu}. Therefore we can rewrite
\begin{align*}
\beta_j^{\lambda/4} \left( e^{-\beta_j \mathcal L} (f_j*k) \right)(u_j)
& = \beta_j^{-(2Q-\lambda)/4} \iint f_j(w) k(w^{-1}v) G(\beta_j^{-1/2} (v^{-1} u_j)) \,dv\,dw \\
& = \beta_j^{Q/4} \iint f_j(u_j (\sqrt{\beta_j}w)) k(w^{-1}v) G(v^{-1} ) \,dv\,dw \\
& = \int g_j(w) H(w) \,dw
\end{align*}
with $g_j(w):= \beta_j^{Q/4} f_j(u_j (\sqrt{\beta_j}w))$ and $H(w) := \int k(w^{-1}v) G(v^{-1} ) \,dv$. Since $\|g_j\|_2 = \|f_j\|_2$ is bounded, the Banach-Alaoglu theorem implies that (after extracting a subsequence if necessary) $(g_j)$ converges weakly in $L^2(\Hei^n)$ to some $g$. Since $k\in L^{2Q/(Q+\lambda)}_{\mathrm w}$ by Proposition \ref{sr} and since $G\in L^s$ for all $s$, in particular for $s=2Q/(2Q-\lambda)$, we infer from the weak Young inequality that $H\in L^2(\Hei^n)$. Therefore by weak convergence
$$
\left| \int g(w) H(w) \,dw \right| = \lim_{j\to\infty} \left| \int g_j(w) H(w) \,dw \right|
= \lim_{j\to\infty} \left| \beta_j^{\lambda/4} \left( e^{-\beta_j \mathcal L} (f_j*k) \right)(u_j) \right|
\geq \tilde\epsilon \,,
$$
which implies that $g\not\equiv 0$. The remaining assertions now follow from Lemma \ref{rellich}.
\end{proof}

Given Corollary \ref{concomp} the existence of an optimizer of \eqref{eq:exmax} follows as in \cite[Lemma 2.7]{Li}. We include the proof for the sake of completeness.

\begin{proof}[Proof of Proposition \ref{exmax}]
 Let $(f_j)$ be a maximizing sequence normalized by $\|f_j\|_2 = 1$. After translations, dilations and passage to a subsequence Corollary \ref{concomp} allows us to assume that $(f_j)$ converges weakly in $L^2$ to a function $f\not\equiv 0$. Moreover, $(f_j*k)$ converges a.e. to $f*k$.

The weak convergence implies that
\begin{equation}
 \label{eq:weakconv}
1= \|f_j\|_2^2 = \|f\|_2^2 + \|f_j-f\|_2^2 + o(1) \,,
\end{equation}
where $o(1)$ denotes something that goes to zero as $j\to\infty$. On the other hand, the pointwise convergence together with the improved Fatou lemma from \cite{BrLi} implies that
\begin{equation}
 \label{eq:bl}
\|f_j*k\|_q^q = \|f*k\|_q^q + \|(f_j-f)*k\|_q^q + \tilde o(1) \,,
\end{equation}
where, again, $\tilde o(1)\to 0$ as $j\to\infty$. Since for $a,b,c\geq 0$ and $q\geq 2$, $(a^q+b^q+c^q)^{2/q} \leq a^2+b^2+c^2$, we have
\begin{equation}
 \label{eq:exmaxproof}
\|f_j*k\|_q^2 \leq  \|f*k\|_q^2 + \|(f_j-f)*k\|_q^2 + \tilde o(1)^{2/q} \,.
\end{equation}
We now estimate the second term on the right side using \eqref{eq:weakconv} and get
$$
\|(f_j-f)*k\|_q^2 \leq \mathcal C_{n,\lambda}^2 \|f_j-f\|_2^2 
= \mathcal C_{n,\lambda}^2 \left( 1- \|f\|_2^2- o(1) \right) \,.
$$
Letting $j\to\infty$ and noting that the left side of \eqref{eq:exmaxproof} converges to $\mathcal C_{n,\lambda}^2$, we conclude that $0\leq \|f*k\|_q^2 - \mathcal C_{n,\lambda}^2 \|f\|_2^2$. Since $f\not\equiv 0$, this implies that $f$ is an optimizer.

In order to see that the convergence of $(f_j)$ in $L^2(\Hei^n)$ is strong, we need to show that $\|f\|_2 =1$. Assume that this is not the case. Then by weak convergence and \eqref{eq:weakconv}, $m:=\|f\|_2^2 \in (0,1)$ and $\lim \|f_j-f\|_2^2=1-m$. Hence by \eqref{eq:bl}
$$
\mathcal C_{n,\lambda}^q = \limsup \|f_j*k\|_q^q \leq \mathcal C_{n,\lambda}^q \left( m^{q/2} + (1-m)^{q/2} \right) \,.
$$
Since $m^{q/2} + (1-m)^{q/2}<1$ for $m\in (0,1)$ we arrive at a contradiction. This completes the proof of Proposition \ref{exmax}.
\end{proof}



\section{Proof of the main theorems}\label{sec:proofs}

Our goal in this section is to compute the sharp constant in inequality \eqref{eq:mainsph} on the sphere $\Sph^{2n+1}$. We shall proceed as in the proof of Theorem \ref{mainjl}. We outline the proof in Subsection \ref{sec:strategy} and reduce everything to the proof of a \emph{linear} inequality. After some preparations in Subsection \ref{sec:funk} we shall prove this inequality in Subsection~\ref{sec:mainproof}.


\subsection{Strategy of the proof}\label{sec:strategy}

\emph{Step 1.} The optimization problem corresponding to \eqref{eq:mainsph} admits an optimizing pair with $f=g$. This has been shown in Subsection \ref{sec:exmax} for inequality \eqref{eq:main} on the Heisenberg group. The result for the inequality on the sphere follows via the Cayley transform as explained in Appendix \ref{sec:equiv}.

We claim that any optimizer for problem \eqref{eq:mainsph} with $f=g$ is a complex multiple of a non-negative function. Indeed,
if we denote the double integral on the left side of \eqref{eq:mainsph} with $g=f$ by $I[f]$ and if $f=a+ib$ for real functions $a$ and $b$, then $I[f]=I[a]+I[b]$. Moreover, for any numbers $\alpha,\beta,\gamma,\delta\in\R$ one has $\alpha\gamma+\beta\delta \leq \sqrt{\alpha^2+\beta^2}\sqrt{\gamma^2+\delta^2}$ with strict inequality unless $\alpha\gamma+\beta\delta\geq 0$ and $\alpha\delta=\beta\gamma$. Since the kernel $|1-\zeta\cdot\overline\eta|^{-\lambda/2}$ is strictly positive, we infer that $I[a]+I[b]\leq I[\sqrt{a^2+b^2}]$ for any real functions $a,b$ with strict inequality unless $a(x)a(y)+b(x)b(y)\geq 0$ and $a(x)b(y)=a(y)b(x)$ for almost every $x,y\in\R^N$. From this one easily concludes that any optimizer is a complex multiple of a non-negative function.

We denote a non-negative optimizer for problem \eqref{eq:mainsph} by $h:=f=g$. Since $h$ satisfies the Euler-Lagrange equation
$$
\int_{\Sph^{2n+1}} \frac{h(\eta)}{{|1-\zeta\cdot\overline\eta|^{\lambda/2}}} \,d\eta = c\, h^{p-1}(\zeta) \,,
$$
we see that $h$ is \emph{strictly} positive.

\emph{Step 2.} As in the proof of Theorem \ref{mainjl}, we may assume that the center of mass of $h^p$ vanishes, that is,
\begin{equation}
 \label{eq:com}
\int_{\Sph^{2n+1}} \zeta_j \, h(\zeta)^p \,d\zeta = 0
\qquad \text{for} \ j=1,\ldots,n+1 \,.
\end{equation}
We shall prove that the only non-negative optimizer satisfying \eqref{eq:com} is the constant function. Then, for exactly the same reason as in the proof of Theorem \ref{mainjl}, the only optimizers without condition \eqref{eq:com} are the ones stated in \eqref{eq:optsph}.

\emph{Step 3.} The second variation around the optimizer $h$ shows that
\begin{equation}
 \label{eq:secvar}
\iint \frac{\overline{f(\zeta)}\ f(\eta)}{|1-\zeta\cdot\overline\eta|^{\lambda/2}} \,d\zeta\,d\eta 
\ \int h^{p} \,d\zeta
- (p-1) \iint \frac{h(\zeta)\ h(\eta)}{|1-\zeta\cdot\overline\eta|^{\lambda/2}} \,d\zeta\,d\eta
\ \int h^{p-2} |f|^2 \,d\zeta \leq 0
\end{equation}
for any $f$ satisfying $\int h^{p-1} f \,d\zeta =0$. Note that the term $h^{p-2}$ causes no problems (despite the fact that $p<2$) since $h$ is strictly positive. In order to prove \eqref{eq:secvar} we proceed as in \eqref{eq:secvarjl}, considering real and imaginary perturbations separately.

Because of \eqref{eq:com} the functions $f(\zeta)=\zeta_j h(\zeta)$ and $f(\zeta)=\overline{\zeta_j} h(\zeta)$ satisfy the constraint $\int h^{p-1} f \,d\zeta =0$. Inserting them in \eqref{eq:secvar} and summing over $j$ we find
\begin{equation}
 \label{eq:secvarineq}
\iint \frac{h(\zeta)\ \left( \zeta\cdot\overline{\eta} + \overline{\zeta}\cdot\eta \right) \ h(\eta)}{|1-\zeta\cdot\overline\eta|^{\lambda/2}} \,d\zeta\,d\eta - 2(p-1) \iint \frac{h(\zeta)\ h(\eta)}{|1-\zeta\cdot\overline\eta|^{\lambda/2}} \,d\omega\,d\eta \leq 0 \,.
\end{equation}

\emph{Step 4.} This is the crucial step! The proof of Theorem \ref{mainsph} is completed by showing that for \emph{any} (not necessarily maximizing) $h$ the \emph{opposite} inequality to \eqref{eq:secvarineq} holds and is indeed strict unless the function is constant. This is the statement of the following theorem with $\alpha=\lambda/4$, noting that $2(p-1) = 2\alpha/(n+1-\alpha)$.

\begin{theorem}\label{keyineq}
 Let $0<\alpha<(n+1)/2$. For any $f$ on $\Sph^{2n+1}$ one has
\begin{equation}\label{eq:keyineq}
 \iint \frac{\overline{f(\zeta)}\ \left( \zeta\cdot\overline{\eta} + \overline\zeta\cdot\eta \right) \ f(\eta)}{|1-\zeta\cdot\overline\eta|^{2\alpha}} \,d\zeta\,d\eta 
\geq \frac{2\alpha}{n+1-\alpha} \iint \frac{\overline{f(\zeta)}\ f(\eta)}{|1-\zeta\cdot\overline\eta|^{2\alpha}} \,d\zeta\,d\eta
\end{equation}
with equality iff $f$ is constant.
\end{theorem}

This theorem will be proved in Subsection \ref{sec:mainproof}.


\subsection{The Funk-Hecke theorem on the complex sphere}\label{sec:funk}

Let $n\in\N$. As before, we consider the sphere $\Sph^{2n+1}$ as a subset of $\C^{n+1}$ and denote coordinates by $(\zeta_1,\ldots,\zeta_{n+1})$ and (non-normalized) measure by $d\zeta$. It is well known that $L^2(\Sph^{2n+1})$ can be decomposed into its $U(n+1)$-irreducible components,
\begin{equation}\label{eq:decomp}
L^2(\Sph^{2n+1})= \bigoplus_{j,k\geq 0} \mathcal H_{j,k} \,.
\end{equation}
The space $\mathcal H_{j,k}$ is the space of restrictions to $\Sph^{2n+1}$ of harmonic polynomials $p(z,\overline z)$ on $\C^{n+1}$ which are homogeneous of degree $j$ in $z$ and degree $k$ in $\overline z$; see \cite{Fo} and references therein.

We shall prove that integral operators whose kernels have the form $K(\zeta\cdot\overline\eta)$ are diagonal with respect to this decomposition and we give an explicit formula for their eigenvalues. In order to state this formula we need the Jacobi polynomials $P_m^{(\alpha,\beta)}$, see \cite[Chapter 22]{AbSt}.

\begin{proposition}\label{funk}
 Let $K$ be an integrable function on the unit ball in $\C$. Then the operator on $\Sph^{2n+1}$ with kernel $K(\zeta\cdot\overline\eta)$ is diagonal with respect to decomposition \eqref{eq:decomp}, and on the space $\mathcal H_{j,k}$ its eigenvalue is given by
\begin{equation}
 \label{eq:funk}
\begin{split}
\frac{\pi^n m!}{2^{n+|j-k|/2} (m+n-1)!}
 & \int_{-1}^1 dt \, (1-t)^{n-1} (1+t)^{|j-k|/2} P_m^{(n-1,|j-k|)}(t) \\
& \qquad\qquad \times \int_{-\pi}^\pi d\phi \, K(e^{-i\phi}\sqrt{(1+t)/2}) \ e^{i(j-k)\phi}   \,,
\end{split}
\end{equation}
where $m:=\min\{j,k\}$.
\end{proposition}

\begin{proof}
 The fact that the operator is diagonal follows from Schur's lemma and the irreducibility of the spaces $\mathcal H_{j,k}$. Now we fix $j$ and $k$ and denote the corresponding eigenvalue by $\lambda$. The projection onto $\mathcal H_{j,k}$ is known (see \cite{Fo} for references) to be the integral operator with kernel $\Phi_{j,k}(\zeta\cdot\overline\eta)$, where
$$
\Phi_{j,k}(r e^{i\phi}) := \frac{(M+n-1)!\ (j+k+n)}{|\Sph^{2n+1}| \ n!\ M!} \, r^{|j-k|} \, e^{i(j-k)\phi} \, P_m^{(n-1,|j-k|)}(2r^2-1)
$$
and $m:=\min\{j,k\}$ and $M:=\max\{j,k\}$. In particular, if $Y_{j,k,\mu}$ denotes an orthonormal basis of $\mathcal H_{j,k}$, then
$$
\sum_\mu Y_{j,k,\mu}(\zeta) \overline{ Y_{j,k,\mu}(\eta) } = \Phi_{j,k}(\zeta\cdot\overline\eta) \,.
$$
Hence multiplying the equation 
$$
\int K(\zeta\cdot\overline\eta) \ Y_{j,k,\mu}(\eta) \,d\eta
= \lambda Y_{j,k,\mu}(\zeta)
$$
by $\overline{Y_{j,k,\mu}(\zeta)}$ and summing over $\mu$ gives
$$
\int K(\zeta\cdot\overline\eta) \ \Phi_{j,k}(\overline\zeta\cdot\eta) \,d\eta
= \lambda \Phi_{j,k}(1) \,.
$$
The left side is independent of $\zeta$, since the right side is so, and hence we can assume that $\zeta=(0,\ldots,0,1)$. We arrive at
\begin{equation}\label{eq:ev}
\lambda = \Phi_{j,k}(1)^{-1} \int K(\overline{\eta_{n+1}}) \ \Phi_{j,k}(\eta_{n+1}) \,d\eta \,.
\end{equation}
In order to simplify this expression, we parametrize $\eta\in\Sph^{2n+1}$ as
$$
\eta = (e^{i\phi_1}\omega_1,\ldots, e^{i\phi_{n+1}}\omega_{n+1})
$$
where $-\pi\leq\phi_j\leq\pi$ and $\omega\in\Sph^n$ with $\omega_j\geq 0$. Now we can parametrize $\omega$ as usual,
\begin{align*}
\omega_1 & =\sin\theta_n \cdots \sin\theta_2 \sin\theta_1  \,,\\
\omega_2 & =\sin\theta_n \cdots \sin\theta_2 \cos\theta_1 \,,\\
\omega_j & = ...\,,\\
\omega_n & = \sin\theta_n \cos\theta_{n-1} \,,\\
\omega_{n+1} & = \cos\theta_n \,,
\end{align*}
with angles satisfying $0\leq\theta_j\leq \pi/2$. In this notation \cite[(11.1.8.1)]{ViKl}
$$
d\eta = d\phi_1\cdots d\phi_{n+1} \sin\theta_1\cos\theta_1 \,d\theta_1 \cdots
\sin^{2n-1}\theta_n \cos\theta_n \,d\theta_n \,.
$$
With this parametrization formula \eqref{eq:ev} becomes
\begin{align*}
\lambda & = \frac{|\Sph^{2n-1}|}{\Phi_{j,k}(1)} \int_0^{\pi/2} \,d\theta \sin^{2n-1}\theta \cos\theta \int_{-\pi}^\pi \,d\phi \, K(e^{-i\phi} \cos\theta) \ \Phi_{j,k}(e^{i\phi} \cos\theta) \\
& = \frac{2 \pi^n m!}{(m+n-1)!}
\int_0^{\pi/2} d\theta \,\sin^{2n-1}\theta \cos^{|j-k|+1}\theta \\
& \qquad\qquad\qquad\qquad\times
\int_{-\pi}^\pi d\phi \, K(e^{-i\phi}\cos\theta) \ e^{i(j-k)\phi}  P_m^{(n-1,|j-k|)}(2\cos^2\theta-1) \,.
\end{align*}
Here we used the explicit expression for $\Phi_{j,k}$, the fact that $|\Sph^{2n-1}|= 2\pi^n/(n-1)!\,$ as well as
$P_m^{(\alpha,\beta)}(1)=\Gamma(m+\alpha+1)/(m! \ \Gamma(\alpha+1))$ \cite[(22.2.1)]{AbSt}. The lemma now follows by the change of variables $t=2\cos^2\theta -1$.
\end{proof}


The Funk-Hecke formula from Proposition \ref{funk} allows us to compute the eigenvalues of two particular families of operators.

\begin{corollary}\label{ev}
Let $-1<\alpha<(n+1)/2$.
\begin{enumerate}
 \item 
The eigenvalue of the operator with kernel $|1-\zeta\cdot\overline\eta|^{-2\alpha}$ on the subspace $\mathcal H_{j,k}$ is
\begin{equation}
 \label{eq:ev1}
E_{j,k} := \frac{2\pi^{n+1} \Gamma(n+1-2\alpha)}{\Gamma^2(\alpha)}
\frac{\Gamma(j+\alpha)}{\Gamma(j+n+1-\alpha)}
\frac{\Gamma(k+\alpha)}{\Gamma(k+n+1-\alpha)} \,.
\end{equation}
\item
The eigenvalue of the operator with kernel $|\zeta\cdot\overline\eta|^2 |1-\zeta\cdot\overline\eta|^{-2\alpha}$ on the subspace $\mathcal H_{j,k}$ is
\begin{equation}
 \label{eq:ev2}
 E_{j,k} \left( 1- \frac{(\alpha-1)(n+1-2\alpha) \left( 2jk+n(j+k-1+\alpha)\right) }{(j-1+\alpha)(j+n+1-\alpha)(k-1+\alpha)(k+n+1-\alpha)} \right) \,.
\end{equation}
\end{enumerate}
When $\alpha=0$ or $1$, formulas \eqref{eq:ev1} and \eqref{eq:ev2} are to be understood by taking limits with fixed $j$ and $k$.
\end{corollary}

Part (1) of this corollary is well-known. It is proved in \cite{JoWa} and \cite{BrFoMo} by different arguments.

\begin{proof}
By Proposition \ref{funk} we have to evaluate the double integral \eqref{eq:funk} for the two choices $K(z) = |1-z|^{-2\alpha}$ and $K(z) = |z|^2 |1-z|^{-2\alpha}$. Our calculations will be based on three formulas, namely the Gamma function identity \cite[(15.1.1) and (15.1.20)]{AbSt}
\begin{align}\label{eq:hypergeom}
 \sum_{\mu=0}^\infty \frac{\Gamma(a+\mu)\ \Gamma(b+\mu)}{\mu!\ \Gamma(c+\mu)}
= \frac{\Gamma(a)\ \Gamma(b)\ \Gamma(c-a-b)}{\Gamma(c-a)\ \Gamma(c-b)}
\end{align}
for $c>a+b$, the cosine integral
\begin{align}\label{eq:gegenint}
 \int_{-\pi}^\pi \!d\phi \left(1-2r\cos\phi+r^2 \right)^{-\alpha} e^{i(j-k)\phi} 
= \frac{2\pi}{\Gamma^2(\alpha)} \sum_{\mu=0}^\infty r^{|j-k|+2\mu} \frac{\Gamma(\alpha+\mu)\ \Gamma(\alpha+|j-k|+\mu)}{\mu! \ (|j-k|+\mu)!  }
\end{align}
for $0\leq r<1$, and the Jacobi polynomial integral
\begin{align}
 \label{eq:jacobiint}
 & \int_{-1}^1 dt \, (1-t)^{n-1} (1+t)^{|j-k|+\mu} P_m^{(n-1,|j-k|)}(t) \\
& \qquad =
\begin{cases}
 0 & \text{if}\ \mu<m \,, \\
2^{|j-k|+n+\mu} \frac{\mu!}{m! \ (\mu-m)!} \frac{(|j-k|+\mu)! \ (m+n-1)!}{(|j-k|+m+n+\mu)!}
& \text{if}\ \mu\geq m \,.
\end{cases}
\notag
\end{align}

Formula \eqref{eq:jacobiint} follows easily from
$$
P_m^{(n-1,|j-k|)}(t) = \frac{(-1)^m}{2^m m!} (1-t)^{-n+1} (1+t)^{-|j-k|} \frac{d^m}{dt^m}\left( (1-t)^{n-1+m} (1+t)^{|j-k|+m}\right) \,;
$$
see \cite[(22.11.1)]{AbSt}. In order to see \eqref{eq:gegenint}, we use the generating function identity for Gegenbauer polynomials,
$$
\left(1-2r\cos\phi+r^2 \right)^{-\alpha} = \sum_{l=0}^\infty C_l^{(\alpha)}(\cos\phi) r^l
$$
and find
$$
\int_{-\pi}^\pi d\phi \, \left(1-2r\cos\phi+r^2 \right)^{-\alpha} \ e^{i(j-k)\phi}
= 2 \sum_{l=0}^\infty r^l  \int_0^\pi d\phi \, C_l^{(\alpha)}(\cos\phi) \cos(j-k)\phi \,.
$$
For fixed $l$ one can evaluate the $\phi$-integral using \cite[(22.3.12)]{AbSt}
\begin{equation*}
C_l^{(\alpha)}(\cos\phi) = \sum_{\nu=0}^l \frac{\Gamma(\alpha+\nu)\ \Gamma(\alpha+l-\nu)}{\nu! \ (l-\nu)! \ \Gamma^2(\alpha) } \cos(l-2\nu)\phi \,,
\end{equation*}
which leads to \eqref{eq:gegenint}.

Up to this point we have verified \eqref{eq:gegenint} and \eqref{eq:jacobiint}. Now we are ready to compute \eqref{eq:funk} with $K(z) = |1-z|^{-2\alpha}$. Using \eqref{eq:gegenint} with $r= \sqrt{(1+t)/2}$, interchanging the $\mu$-sum with the integral and doing the $t$-integration using \eqref{eq:jacobiint}, we obtain
\begin{align*}
E_{j,k} & = \frac{\pi^n m!}{2^{|j-k|+n+\mu} (m+n-1)!} \frac{2\pi}{\Gamma^2(\alpha)}
\sum_{\mu=0}^\infty  \frac{\Gamma(\alpha+\mu)\ \Gamma(|j-k|+\alpha+\mu)}{\mu! \ (|j-k|+\mu)!  } \\
& \qquad\qquad \times \int_{-1}^1 dt \, (1-t)^{n-1} (1+t)^{|j-k|+\mu} P_m^{(n-1,|j-k|)}(t) \\
& = \frac{2\pi^{n+1}}{\Gamma^2(\alpha)}
\sum_{\mu=m}^\infty  \frac{\Gamma(\alpha+\mu)\ \Gamma(|j-k|+\alpha+\mu)}{(\mu-m)! \ (|j-k|+m+n+\mu)!} \\
& = \frac{2\pi^{n+1}}{\Gamma^2(\alpha)}
\sum_{\mu=0}^\infty  \frac{\Gamma(m+\alpha+\mu)\ \Gamma(|j-k|+m+\alpha+\mu)}{\mu! \ (|j-k|+2m+n+\mu)!} \\
& = \frac{2\pi^{n+1}}{\Gamma^2(\alpha)}
\sum_{\mu=0}^\infty  \frac{\Gamma(j+\alpha+\mu)\ \Gamma(k+\alpha+\mu)}{\mu! \ (j+k+n+\mu)!} \\
& = \frac{2\pi^{n+1} \Gamma(n+1-2\alpha)}{\Gamma^2(\alpha)}
\frac{\Gamma(j+\alpha)}{\Gamma(j+n+1-\alpha)}
\frac{\Gamma(k+\alpha)}{\Gamma(k+n+1-\alpha)} \,.
\end{align*}
The last identity used \eqref{eq:hypergeom}.

The computation in the case $K(z)= |z|^2 |1-z|^{-2\alpha}$ is similar but more complicated. The extra factor $|z|^2$ introduces an extra factor $(1+t)/2$ in the $t$ integral. After doing the $\phi$ and the $t$ integral using \eqref{eq:gegenint} and \eqref{eq:jacobiint} we arrive at
\begin{align*}
& \frac{\pi^n m!}{2^{|j-k|+n+1+\mu} (m+n-1)!} \frac{2\pi}{\Gamma^2(\alpha)}
\sum_{\mu=0}^\infty  \frac{\Gamma(\alpha+\mu)\ \Gamma(|j-k|+\alpha+\mu)}{\mu! \ (|j-k|+\mu)!  } \\
& \qquad \qquad \times \int_{-1}^1 dt \, (1-t)^{n-1} (1+t)^{|j-k|+1+\mu} P_m^{(n-1,|j-k|)}(t) \\
& \qquad = \frac{2\pi^{n+1} }{\Gamma^2(\alpha)} 
\sum_{\mu=\max\{m-1,0\}}^\infty \frac{(\mu+1)\ (|j-k|+1+\mu)\ \Gamma(\alpha+\mu)\ \Gamma(|j-k|+\alpha+\mu)}{(\mu+1-m)! \ (|j-k|+m+n+1+\mu)!} \,.
\end{align*}
Now we distinguish two cases according to whether $m=0$ or not. In the first case, the above equals
\begin{align*}
& \frac{2\pi^{n+1} }{\Gamma^2(\alpha)} 
\sum_{\mu=0}^\infty  (|j-k|+1+\mu) \, \frac{\Gamma(\alpha+\mu)\ \Gamma(|j-k|+\alpha+\mu)}{\mu! \ (|j-k|+n+1+\mu)!} \\
& = \frac{2\pi^{n+1} }{\Gamma^2(\alpha)} 
\sum_{\mu=0}^\infty  \left( \frac{\Gamma(\alpha+\mu)\ \Gamma(|j-k|+\alpha+\mu)}{\mu! \ (|j-k|+n+\mu)!} 
- \frac{n\ \Gamma(\alpha+\mu)\ \Gamma(|j-k|+\alpha+\mu)}{\mu! \ (|j-k|+n+1+\mu)!} \right) \,.
\end{align*}
Because of \eqref{eq:hypergeom} this is equal to
\begin{align*}
& \frac{2\pi^{n+1} }{\Gamma^2(\alpha)} \!\left(\! \frac{\Gamma(\alpha)\ \Gamma(|j-k|+\alpha)\ \Gamma(n+1-2\alpha) }{\Gamma(|j-k|+n+1-\alpha)\ \Gamma(n+1-\alpha)}
\!-\! \frac{n\ \Gamma(\alpha)\ \Gamma(|j-k|+\alpha)\ \Gamma(n+2-2\alpha)}{\Gamma(|j-k|+n+2-\alpha)\ \Gamma(n+2-\alpha)} \right) \\
& = E_{j,k} \left( 1 - \frac{n (n+1-2\alpha)}{(|j-k|+n+1-\alpha)(n+1-\alpha)} \right) \,,
\end{align*}
which coincides with the claimed expression.

In the case $m\geq 1$, the eigenvalue is given by
\begin{align*}
 & \frac{2\pi^{n+1} }{\Gamma^2(\alpha)} 
\sum_{\mu=m-1}^\infty  \frac{(\mu+1)(|j-k|+1+\mu) \ \Gamma(\alpha+\mu)\ \Gamma(|j-k|+\alpha+\mu)}{(\mu+1-m)! \ (|j-k|+m+n+1+\mu)!} \\
& = \frac{2\pi^{n+1} }{\Gamma^2(\alpha)} 
\sum_{\mu=0}^\infty  \frac{(\mu+m)(|j-k|+m+\mu) \, \Gamma(m-1+\alpha+\mu)\ \Gamma(|j-k|+m-1+\alpha+\mu)}{\mu! \ (|j-k|+2m+n+\mu)!} \\
& = \frac{2\pi^{n+1} }{\Gamma^2(\alpha)} 
\sum_{\mu=0}^\infty  \frac{(j+\mu)(k+\mu) \, \Gamma(j-1+\alpha+\mu)\ \Gamma(k-1+\alpha+\mu)}{\mu! \ (j+k+n+\mu)!} \\
& = \frac{2\pi^{n+1} }{\Gamma^2(\alpha)} 
\sum_{\mu=0}^\infty  \frac{1}{\mu! \ (j+k+n+\mu)!}
\Big\{ \Gamma(j+\alpha+\mu)\ \Gamma(k+\alpha+\mu) \\
& \qquad - (\alpha-1) \big[ \Gamma(j-1+\alpha+\mu)\ \Gamma(k+\alpha+\mu) + \Gamma(j+\alpha+\mu)\ \Gamma(k-1+\alpha+\mu)\big] \\
& \qquad + (\alpha-1)^2 \Gamma(j-1+\alpha+\mu)\ \Gamma(k-1+\alpha+\mu) \Big\} \,.
\end{align*}
Once again, we use \eqref{eq:hypergeom} in order to simplify the sum and we obtain
\begin{align*}
& \frac{2\pi^{n+1} }{\Gamma^2(\alpha)}
\Bigg\{ \frac{\Gamma(j+\alpha)\ \Gamma(k+\alpha)\ \Gamma(n+1-2\alpha)}{\Gamma(j+n+1-\alpha)\ \Gamma(k+n+1-\alpha)} \\
& \qquad \qquad - (\alpha-1) \Bigg[ \frac{\Gamma(j-1+\alpha)\ \Gamma(k+\alpha)\ \Gamma(n+2-2\alpha)}{\Gamma(j+n+1-\alpha)\ \Gamma(k+n+2-\alpha)} \\
& \qquad \qquad \qquad \qquad\qquad
+ \frac{\Gamma(j+\alpha)\ \Gamma(k-1+\alpha)\ \Gamma(n+2-2\alpha)}{\Gamma(j+n+2-\alpha)\ \Gamma(k+n+1-\alpha)}
\Bigg] \\
& \qquad\qquad + (\alpha-1)^2 \frac{\Gamma(j-1+\alpha)\ \Gamma(k-1+\alpha)\ \Gamma(n+3-2\alpha)}{\Gamma(j+n+2-\alpha)\ \Gamma(k+n+2-\alpha)} \Bigg\} \,,
\end{align*}
which can be simplified to the claimed form \eqref{eq:ev2}.
\end{proof}


\subsection{Proof of Theorem \ref{keyineq}}\label{sec:mainproof}

Using that
$$
|1-\zeta\cdot\overline\eta|^2 = 1 - (\zeta\cdot\overline\eta+\overline\zeta\cdot\eta) + |\zeta\cdot\overline\eta|^2 \,,
$$
we see that it is equivalent to prove
\begin{align*}
 & \iint \overline{f(\zeta)} 
\left( \frac{1 - \left|\zeta\cdot\overline{\eta}\right|^2}{|1-\zeta\cdot\overline\eta|^{2\alpha}} 
+ \frac{1}{|1-\zeta\cdot\overline\eta|^{2(\alpha-1)}}\right) f(\eta)\,d\zeta\,d\eta \\
& \qquad \leq \frac{2(n+1-2\alpha)}{n+1-\alpha} \iint \frac{\overline{f(\zeta)}\ f(\eta)}{|1-\zeta\cdot\overline\eta|^{2\alpha}} \,d\zeta\,d\eta \,.
\end{align*}
Both quadratic forms are diagonal with respect to decomposition \eqref{eq:decomp} and their eigenvalues on the subspace $\mathcal H_{j,k}$ are given by Corollary \ref{ev}. For simplicity, we first assume that $\alpha\neq 1$. The eigenvalue of the right side is $2(n+1-2\alpha)E_{j,k}/(n+1-\alpha)$, with $E_{j,k}$ given by \eqref{eq:ev1}, and the eigenvalue of the left side is
$$
E_{j,k} \frac{(\alpha-1)(n+1-2\alpha) \left( 2jk+n(j+k-1+\alpha)\right) }{(j-1+\alpha)(j+n+1-\alpha)(k-1+\alpha)(k+n+1-\alpha)}  + \tilde E_{j,k} \,,
$$
where $\tilde E_{j,k}$ is $E_{j,k}$ with $\alpha$ replaced by $\alpha-1$. Noting that
$$
\tilde E_{j,k} = E_{j,k} \frac{(\alpha-1)^2 (n+1-2\alpha)(n+2-2\alpha)}{(j-1+\alpha)(j+n+1-\alpha)(k-1+\alpha)(k+n+1-\alpha)}
$$
and that $E_{j,k}>0$, we see that the conclusion of the theorem is equivalent to the inequality
\begin{align*}
& \frac{(\alpha-1)(n+1-2\alpha) \left( 2jk+n(j+k-1+\alpha)\right)
+ (\alpha-1)^2 (n+1-2\alpha)(n+2-2\alpha) }{(j-1+\alpha)(j+n+1-\alpha)(k-1+\alpha)(k+n+1-\alpha)} \\ 
& \qquad\leq \frac{2(n+1-2\alpha)}{n+1-\alpha}
\end{align*}
for all $j,k\geq 0$. Since $\alpha<(n+1)/2$, this is the same as
$$
\frac{(\alpha-1)\left( 2jk+n(j+k) +2(\alpha-1)(n+1-\alpha)\right)}{(j-1+\alpha)(j+n+1-\alpha)(k-1+\alpha)(k+n+1-\alpha)} \,
\leq \frac{2}{n+1-\alpha}
$$
or, equivalently,
$$
(\alpha-1)\left( \frac{1}{(j-1+\alpha)(k+n+1-\alpha)} + \frac{1}{(j+n+1-\alpha)(k-1+\alpha)} \right)
\leq \frac{2}{n+1-\alpha} \,.
$$
This inequality is elementary to prove, distinguishing the cases $\alpha>1$ and $\alpha<1$. Finally, the case $\alpha=1$ is proved by letting $\alpha\to 1$ for fixed $j$ and $k$.

Strictness of inequality \eqref{eq:keyineq} for non-constant $f$ follows from the fact that the above inequalities are strict unless $j=k=0$. This completes the proof of Theorem~\ref{keyineq}.
\qed


\subsection{Proof of Sobolev inequalities on the sphere}\label{sec:sobq}

\begin{proof}[Proof of Corollary \ref{sobq}]
We define the number $d\in(0,2)$ by $q=2Q/(Q-d)$ and the operator $A_d$ in $L^2(\Sph^{2n+1})$ which acts as multiplation on $\mathcal H_{j,k}$ by
$$
\frac{\Gamma(\tfrac{Q+d}4+j)\ \Gamma(\tfrac{Q+d}4+k)}{\Gamma(\tfrac{Q-d}4+j)\ \Gamma(\tfrac{Q-d}4+k)} \,.
$$
The same duality argument that relates the $\lambda=Q-2$ case of \eqref{eq:mainsph} to \eqref{eq:jlsph} relates the $\lambda=Q-d$ case to the inequality
\begin{equation}
 \label{eq:hlsdual}
(u, A_d u) 
\geq |\Sph^{2n+1}|^{1-2/q} \ \frac{\Gamma(\tfrac{Q+d}4)^2}{\Gamma(\tfrac{Q-d}4)^2} \left( \int_{\Sph^{2n+1}} |u|^q \,d\zeta \right)^{2/q} \,;
\end{equation}
see \cite{BrFoMo} for details. (This can also be obtained from Part (1) of Corollary \ref{ev}.) Hence the claimed inequality will follow if we can prove that
$$
\frac{8d}{(Q-d)(Q-2)} \mathcal E[u] +  \int_{\Sph^{2n+1}} |u|^2 \,d\zeta 
\geq \frac{\Gamma(\tfrac{Q-d}4)^2}{\Gamma(\tfrac{Q+d}4)^2} (u, A_d u) \,.
$$
Since $\mathcal L$ acts on $\mathcal H_{j,k}$ as multiplication by $jk+\tfrac{Q-2}4(j+k)$, this inequality is equivalent to
\begin{equation}\label{eq:goal}
\frac{8d}{(Q-d)(Q-2)} \left( jk+\tfrac{Q-2}4(j+k) \right) + 1 
\geq \frac{\Gamma(\tfrac{Q-d}4)^2}{\Gamma(\tfrac{Q+d}4)^2} \frac{\Gamma(\tfrac{Q+d}4+j)\ \Gamma(\tfrac{Q+d}4+k)}{\Gamma(\tfrac{Q-d}4+j)\ \Gamma(\tfrac{Q-d}4+k)} \,.
\end{equation}
We first prove this inequality for $j=0$, that is, we first show that
\begin{equation}
 \label{eq:j0}
\frac{2d}{Q-d} k + 1
\geq \frac{\Gamma(\tfrac{Q-d}4)}{\Gamma(\tfrac{Q+d}4)} \frac{\Gamma(\tfrac{Q+d}4+k)}{\Gamma(\tfrac{Q-d}4+k)} \,.
\end{equation}
This inequality is proved in \cite[p. 233]{Be2}, but for later reference we reproduce part of the argument. Since \eqref{eq:j0} is an equality at $k=0$ and $k=1$, we only need to prove that the logarithmic derivative with respect to $k$ of the left side is greater than or equal to that of the right side for any $k\geq 1$, that is,
\begin{equation}
 \label{eq:digamma}
\frac{2d}{2d k + Q-d} \geq \psi(\tfrac{Q+d}4+k)-\psi(\tfrac{Q-d}4+k) \,,
\end{equation}
where $\psi=(\ln\Gamma)'$ is the digamma function. This follows from \cite[(38)]{Be2} with $n$ and $q$ in \cite{Be2} replaced by our $Q/2$ and $2Q/(Q-d)$, respectively.

Having proved \eqref{eq:goal} for $j=0$, we shall now prove that the logarithmic derivative of the left side with respect to $j$ is greater than or equal to that of the right side for any $j\geq 1$ and $k\geq 0$, that is,
$$
\frac{8d \left( k+\tfrac{Q-2}4 \right)}{8d \left( jk+\tfrac{Q-2}4(j+k) \right) + (Q-d)(Q-2) }
\geq \psi(\tfrac{Q+d}4+j)-\psi(\tfrac{Q-d}4+j) \,.
$$
Since here the right side is independent of $k$, we can take the infimum of the left side over $k$. Using the fact that $d\leq 2$ one easily sees that the left side is increasing with respect to $k\geq 0$. Hence we only need to prove the inequality with $k=0$,
$$
\frac{2d }{2d j + Q-d} \geq \psi(\tfrac{Q+d}4+j)-\psi(\tfrac{Q-d}4+j) \,,
$$
but this is again \eqref{eq:digamma}. This completes the proof of \eqref{eq:goal}.

To show that equality holds \emph{only} for constant functions, we examine the preceeding proof and see that \eqref{eq:goal} is strict unless $(j,k)$ is $(0,0)$, $(0,1)$ or $(1,0)$. However, in the two latter cases \eqref{eq:hlsdual} is strict, as seen from Theorem \ref{mainsph}. This proves the corollary.
\end{proof}


\subsection{Proofs of the endpoint inequalities}

\begin{proof}[Sketch of proof of Corollary \ref{mainlog}]
 The first part of the proof is a by now standard differentiation argument; for some technical details we refer, e.g., to \cite[Thm. 8.14]{LiLo}. We subtract $|\Sph^{2n+1}|^2 = \iint f(\zeta) g(\eta)\,d\zeta\,d\eta$ from each side in \eqref{eq:mainsph} and divide by $\lambda$. In the limit $\lambda\to 0$ we obtain \eqref{eq:mainlog}.

In order to see that the constant $|\Sph^{2n+1}|/Q$ is sharp, we take $f(\zeta)=g(\zeta)=1+\epsilon\re\zeta_1$ as trial functions. After some computations we find that \eqref{eq:mainlog} is an equality up to order $\epsilon^2$ as $\epsilon\to 0$. (A limiting version of Corollary \ref{ev} is helpful for the computation of the integrals.) 
\end{proof}

\begin{proof}[Sketch of proof of Corollary \ref{mainent}]
 Indeed, we first note that the constant $\tilde D_{n,\lambda}$ on the right side of \eqref{eq:mainsph} satisfies
$$
\tilde D_{n,\lambda} \, |\Sph^{2n+1}|^{(2-p)/p} = \int_{\Sph^{2n+1}} \frac{d\zeta}{|1-\zeta_{n+1}|^{\lambda/2}} \,,
$$
and therefore for any non-negative $f$ with $\|f\|_2^2=|\Sph^{2n+1}|$,
$$
\frac12 \iint_{\Sph^{2n+1}\times\Sph^{2n+1}} \frac{f(\zeta)^2 +f(\eta)^2}{|1-\zeta\cdot\overline\eta|^{\lambda/2}} \,d\zeta\,d\eta = \tilde D_{n,\lambda} \, |\Sph^{2n+1}|^{2/p} \,.
$$
Subtracting this from each side of \eqref{eq:mainsph}, we see that the left side of \eqref{eq:mainent} with exponent $Q/2$ replaced by $\lambda/2$ is bounded from below by
$$
2 \tilde D_{n,\lambda} \left( |\Sph^{2n+1}|^{2/p} - \|f\|_p^2 \right) \,.
$$
Inequality \eqref{eq:mainent} now follows by recalling the explicit expression of $\tilde D_{n,\lambda}$ and letting $\lambda\to Q$.

To check that the constant is sharp, we take $f(\zeta) = \sqrt{1-\epsilon^2} + \epsilon\re\zeta_1$ and check (using, e.g., the calculations in Corollary \ref{ev}) that \eqref{eq:mainent} is an equality up to order $\epsilon^2$ as $\epsilon\to 0$.
\end{proof}


\appendix

\section{Equivalence of Theorems \ref{main} and \ref{mainsph}}\label{sec:equiv}

In this appendix we consider the Cayley transform $\mathcal C: \Hei^n \to \Sph^{2n+1}$ and its inverse $\mathcal C^{-1}:  \Sph^{2n+1}\to \Hei^n$ given by
\begin{align*}
 \mathcal C(z,t) & = \left(\frac{2z}{1+|z|^2+it}, \frac{1-|z|^2-it}{1+|z|^2+it} \right) \,, \\
\mathcal C^{-1}(\zeta) & = \left(\frac{\zeta_1}{1+\zeta_{n+1}},\ldots,\frac{\zeta_n}{1+\zeta_{n+1}}, \im \frac{1-\zeta_{n+1}}{1+\zeta_{n+1}} \right) \,.
\end{align*}
The Jacobian of this transformation (see, e.g., \cite{BrFoMo}) is
$$
J_{\mathcal C}(z,t) = \frac{2^{2n+1}}{((1+|z|^2)^2+t^2)^{n+1}} \,,
$$
which implies that
\begin{equation}
 \label{eq:change}
\int_{\Sph^{2n+1}} \phi(\zeta) \,d\zeta = \int_{\Hei^n} \phi(\mathcal C(u)) J_{\mathcal C}(u) \,du
\end{equation}
for any integrable function $\phi$ on $\Sph^{2n+1}$.

We now explain the equivalence of \eqref{eq:main} and \eqref{eq:mainsph}, which depends on $\lambda$ and on $p$, which is related to $\lambda$ by $p=2Q/(2Q-\lambda)$. There is a one-to-one correspondence between functions $f$ on $\Sph^{2n+1}$ and functions $F$ on $\Hei^n$ given by
\begin{equation}
 \label{eq:corres}
F(u)=|J_{\mathcal C}(u)|^{1/p} f(\mathcal C(u)) \,.
\end{equation}
It follows immediately from \eqref{eq:change} that $f\in L^p(\Sph^{2n+1})$ if and only if $F\in L^p(\Hei^n)$, and in this case $\|f\|_p = \|F\|_p$. Moreover, using the fact that
$$
|1-\zeta\cdot\overline\eta| = 2 \left( (1+|z|^2)^2+t^2 \right)^{-1/2} |u^{-1} v|^2 \left( (1+|z'|^2)^2+(t')^2 \right)^{-1/2}
$$
for $\zeta=\mathcal C(u)=\mathcal C(z,t)$ and $\eta=\mathcal C(v) = \mathcal C(z',t')$, one easily verifies that
$$
\iint_{\Hei^n\times\Hei^n} \frac{\overline{F(u)}\ F(v)}{|u^{-1} v|^\lambda} \,du\,dv 
= 2^{-n\lambda/Q} \iint_{\Sph^{2n+1}\times\Sph^{2n+1}} \frac{\overline{f(\zeta)}\ f(\eta)}{|1-\zeta\cdot\overline\eta|^{\lambda/2}} \,d\zeta\,d\eta \,.
$$
This shows that the sharp constants in \eqref{eq:main} and \eqref{eq:mainsph} coincide up to a factor of $2^{-n\lambda/Q}$ and that there is a one-to-one correspondence between optimizers. In particular, the function $f\equiv 1$ on $\Sph^{2n+1}$ corresponds to the function
$$
|J_{\mathcal C}(u)|^{1/p} = 2^{(2n+1)(2Q-\lambda)/4(n+1)} H(u)
$$
on $\Hei^n$ with $H$ given in \eqref{eq:opt}.

Similarly, when $p=2Q/(Q-2)$, and $F$ and $f$ are related via \eqref{eq:corres}, then
\begin{equation}
 \label{eq:corressob}
\begin{split}
& \frac14 \sum_{j=1}^n \int_{\Hei^n} \left( \left| \left( \frac\partial{\partial x_j} +2 y_j \frac\partial{\partial t} \right) F \right|^2 + \left| \left( \frac\partial{\partial y_j} -2 x_j \frac\partial{\partial t} \right) F\right|^2 \right) \,du  \\
&\qquad = 2^{1/(n+1)} \frac12 \int_{\Sph^{2n+1}} \left( \sum_{j=1}^{n+1} \left( |T_j f|^2 + |\overline{T_j} f|^2 \right) + \frac{n^2}2 |f|^2 \right) \,d\zeta \,.
\end{split}
\end{equation}
This can be checked by computation, cf. also \cite{JeLe1,BrFoMo}.


\section{The center of mass condition}\label{sec:com}

Here, we prove that by suitable inequality preserving transformation of $\Sph^{2n+1}$ we may assume the center of mass conditions given in \eqref{eq:comjl} and \eqref{eq:com}.

We shall define a family of maps $\gamma_{\delta,\xi}:\Sph^{2n+1}\to\Sph^{2n+1}$ depending on two parameters $\delta>0$ and $\xi\in\Sph^{2n+1}$. To do so, we denote dilation on $\Hei^n$ by $\mathcal S_\delta$, that is, $\mathcal S_\delta(u)=\delta u$. Moreover, for any $\xi\in\Sph^{2n+1}$ we choose a unitary $(n+1)\times(n+1)$ matrix $U$ such that $U\xi=(0,\ldots,0,1)$ and we put
$$
\gamma_{\delta,\xi}(\zeta) := U^* \mathcal C\left(\mathcal S_\delta\left(\mathcal C^{-1}\left(U\zeta\right)\right)\right)
$$
for all $\zeta\in\Sph^{2n+1}$. This transformation depends only on $\xi$ (and $\delta$) and not on the particular choice of $U$. Indeed, an elementary computation shows that
\begin{equation}
 \label{eq:gamma}
\begin{split}
\gamma_{\delta,\xi}(\zeta) = &\left( \frac{2\delta(1+\xi\cdot\overline{\zeta})}{|1+\zeta\cdot\overline\xi|^2+\delta^2 (1-|\zeta\cdot\overline\xi|^2-2i\im(\zeta\cdot\overline\xi))} \right) \left(\zeta- (\zeta\cdot\overline\xi) \ \xi \right) \\
& + \left( \frac{|1+\zeta\cdot\overline\xi|^2-\delta^2 (1-|\zeta\cdot\overline\xi|^2-2i\im(\zeta\cdot\overline\xi))}
{|1+\zeta\cdot\overline\xi|^2+\delta^2 (1-|\zeta\cdot\overline\xi|^2-2i\im(\zeta\cdot\overline\xi))} \right) \xi \,.
\end{split}
\end{equation}

\begin{lemma}
 \label{com}
Let $f\in L^1(\Sph^{2n+1})$ with $\int_{\Sph^{2n+1}} f(\zeta) \,d\zeta\neq 0$. Then there is a transformation $\gamma_{\delta,\xi}$ of $\Sph^{2n+1}$ such that
$$
\int_{\Sph^{2n+1}} \gamma_{\delta,\xi}(\zeta) f(\zeta) \,d\zeta = 0 \,.
$$
\end{lemma}

\begin{proof}
We may assume that $f\in L^1(\Sph^{2n+1})$ is normalized by $\int_{\Sph^{2n+1}} f(\zeta) \,d\zeta=1$. We shall show that the $\C^{n+1}$-valued function
$$
F(r\xi) := \int_{\Sph^{2n+1}} \gamma_{1-r,\xi}(\zeta) f(\zeta) \,d\zeta\,,
\qquad 0< r< 1\,, \xi\in\Sph^{2n+1} \,, 
$$
has a zero. First, note that because of $\gamma_{1,\xi}(\zeta)= \zeta$ for all $\xi$ and all $\zeta$, the limit of $F(r\xi)$ as $r\to 0$ is independent of $\xi$. In other words, $F$ is a continuous function on the open unit ball of $\R^{2n+2}$. In order to understand its boundary behavior, one easily checks that for any $\zeta\neq-\xi$ one has $\lim_{\delta\to 0} \gamma_{\delta,\xi}(\zeta)= \xi$, and that  this convergence is uniform on $\{(\zeta,\xi) \in \Sph^{2n+1}\times\Sph^{2n+1} :\ |1+\zeta\cdot\overline\xi| \geq \epsilon\}$ for any $\epsilon>0$. This implies that
\begin{equation}
 \label{eq:brouwerlimit}
\lim_{r\to 1} F(r\xi) = \xi
\qquad\text{uniformly in}\ \xi\,.
\end{equation}
Hence $F$ is a continuous function on the \emph{closed} unit ball, which is the identity on the boundary. The assertion is now a consequence of Brouwer's fixed point theorem.
\end{proof}

In the proof of Theorem \ref{mainjl} we use Lemma \ref{com} with $f=|u|^q$. Then the new function $\tilde u(\zeta) = |J_{\gamma^{-1}}(\zeta)|^{1/q} u(\gamma^{-1}(\zeta))$, with $\gamma=\gamma_{\delta,\xi}$ of Lemma \ref{com}, satisfies the center of mass condition \eqref{eq:comjl}. Moreover, since rotations of the sphere, the Cayley transform $\mathcal C$ and the dilations $\mathcal S_{\delta}$ leave the inequality invariant, $u$ can be replaced by $\tilde u$ in \eqref{eq:mainsph} without changing the values of each side.

In particular, if $w$ were an optimizer our proof in Section~\ref{sec:jl} shows that the corresponding $\tilde w$ is a constant, which means that the original $w$ is a constant times $|J_\gamma(\zeta)|^{1/q}$. It is now a matter of computation, which has fortunately been done in \cite[(1.14)]{BrFoMo} (based on \cite{JeLe}), to verify that all such functions have the form of \eqref{eq:optjl}.

Conversely, let us verify that all the functions given in \eqref{eq:optjl} are optimizers. By the rotation invariance of inequality \eqref{eq:mainjl}, we can restrict our attention to the case $\xi= (0,\ldots,0,r)$ with $0<r<1$. These functions correspond via the Cayley transform, \eqref{eq:corres}, to dilations of a constant times the function $H$ in \eqref{eq:opt}. Because of the dilation invariance of inequality \eqref{eq:sobjl} and because of the fact that we already know that $H$, which corresponds to the constant on the sphere, is an optimizer, we conclude that any function of the form \eqref{eq:optjl} is an optimizer.

We have discussed the derivative (Sobolev) version of the $\lambda=Q-2$ case of \eqref{eq:mainsph}. Exactly the same considerations show the invariance of the fractional integral for all $0<\lambda<Q$.


\subsection*{Acknowledgements} We thank Richard Bamler for valuable help with Appendix \ref{sec:com}.


\bibliographystyle{amsalpha}

\end{document}